\documentclass[11pt]{amsart}

 \usepackage{amsmath,amsthm,amsfonts,amssymb,verbatim}
 \usepackage{url}
 \usepackage{graphicx}
 \usepackage[all]{xy}
\usepackage{tikz}\usetikzlibrary{matrix,arrows}
 \usepackage{wrapfig}
\usepackage{picins}
 \usepackage{pinlabel}
 \usepackage{subfigure}
 \usepackage{xfrac}
 
\def\co{\colon\thinspace}

\newcommand{\KHT}{\underset{\raisebox{3pt}{$\longleftarrow$}}{\operatorname{Kh}}}

\newcommand{\phalf}{\text{\sfrac{1}{2}}}
\newcommand{\mhalf}{\text{-\sfrac{1}{2}}}
\newcommand{\miniminus}{\text{-}}
\newcommand{\Br}{\mathbf{\Sigma}}
\newcommand{\boldA}{\mathbf{A}}
\newcommand{\boldK}{\mathbf{K}}

\newcommand{\bu}{\bullet}

\newcommand{\sC}{\mathcal{C}}

\newcommand{\odd}{\mathbb{Z}_{\operatorname{odd}}}

\newcommand{\bZ}{\mathbb{Z}}

\newcommand{\bF}{\mathbb{F}}

\newcommand{\half}{\textstyle\frac{1}{2}}

\newcommand{\into}{\hookrightarrow}

\newcommand{\HFhat}{\widehat{\operatorname{HF}}}

\newcommand{\HFK}{\widehat{\operatorname{HFK}}}

\newcommand{\Khred}{\widetilde{\operatorname{Kh}}{}}

\newcommand{\Fix}{\operatorname{Fix}}

\newcommand{\Sym}{\operatorname{Sym}}

\newcommand{\kh}{\varkappa}

\newtheorem{theorem}{Theorem}

\newtheorem{corollary}[theorem]{Corollary}
\newtheorem{proposition}[theorem]{Proposition}
\newtheorem{lemma}[theorem]{Lemma}
\newtheorem{conjecture}[theorem]{Conjecture}

\newtheorem*{namedtheorem}{\theoremname}
\newcommand{\theoremname}{testing}
\newenvironment{named}[1]{\renewcommand{\theoremname}{#1}
        \begin{namedtheorem}}
        {\end{namedtheorem}}

\theoremstyle{definition}
\newtheorem{definition}[theorem]{Definition}
\newtheorem{question}[theorem]{Question}
\newtheorem{remark}[theorem]{Remark}

\newcommand{\positive}
	{\raisebox{-2pt}{\includegraphics[scale=0.1]{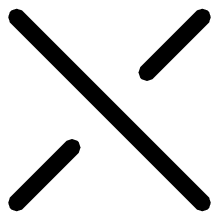}}}

\newcommand{\rightcross}
	{\raisebox{-2pt}
	{\includegraphics[scale=0.1]{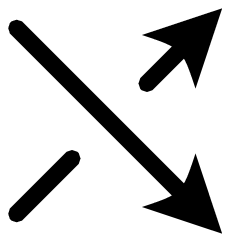}}}

\newcommand{\zero}
	{\raisebox{-2pt}
	{\includegraphics[scale=0.1]{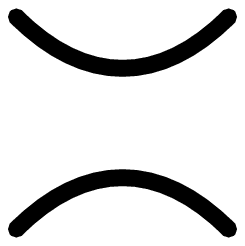}}}
\newcommand{\one}
	{\raisebox{-2pt}
	{\includegraphics[scale=0.1]{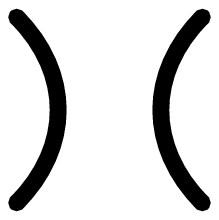}}}

\title[Khovanov homology and the symmetry group]{Khovanov homology\\ and the symmetry group of a knot}
\date{April 5, 2017}

\author[Liam Watson]{Liam Watson}
\thanks{Partially supported by a Marie Curie Career Integration Grant (HFFUNDGRP)}
\address{University of Glasgow, School of Mathematics and Statistics, Glasgow, United Kingdom \quad
{\it Current adress:} Universit\'e de Sherbrooke, D\'epartement de Math\'ematiques, Sherbrooke, Qu\'ebec, Canada
}
\email{liam.watson@usherbrooke.ca}
\urladdr{http://math.usherbrooke.ca/lwatson}

\begin{document}

\begin{abstract} 
We introduce an invariant of tangles in Khovanov homology by considering a natural inverse system of Khovanov homology groups. As application, we derive an invariant of strongly invertible knots; this invariant takes the form of a graded vector space that vanishes if and only if the strongly invertible knot is trivial. While closely tied to Khovanov homology --- and hence the Jones polynomial --- we observe that this new invariant detects non-amphicheirality in subtle cases where Khovanov homology fails to do so. In fact, we exhibit examples of knots that are not distinguished by Khovanov homology but, owing to the presence of a strong inversion, may be distinguished using our invariant. This work suggests a strengthened relationship between Khovanov homology and Heegaard Floer homology by way of two-fold branched covers that we formulate in a series of conjectures.
\end{abstract}

\maketitle

\hfill\begin{footnotesize}{\em To my grandfather, Karl Erik Snider.}\end{footnotesize}

The reduced Khovanov homology of an oriented link $L$ in the three-sphere is a bi-graded vector space $\Khred(L)$  for which the graded Euler characteristic $\sum_{u,q}(-1)^ut^q\dim\Khred^u_q(L)$ recovers the Jones polynomial of $L$ \cite{Khovanov2000,Khovanov2003}. This homological link invariant is known to detect the trivial knot. Precisely, Kronheimer and Mrowka prove that $\dim\Khred(K)=1$ if and only if $K$ is the trivial knot \cite{KM2011}. It remains an open problem to determine if the analogous detection result holds for the Jones polynomial.

Preceding the work of Kronheimer  and Mrowka are a range of applications of Khovanov homology in low-dimensional topology. Perhaps most recognised among these is Rasmussen's combinatorial proof of the Milnor conjecture by way of the $s$ invariant \cite{Rasmussen2010}. Other examples include Ng's bound on the Thurston-Bennequin number \cite{Ng2005}, Plamenevskaya's invariant of transverse knots \cite{Plamenevskaya2006}, and obstructions to finite fillings on strongly invertible knots due to the author \cite{Watson2012,Watson2013}. 

There are two features common to each of these applications. First, the quantity extracted from Khovanov homology is an integer (a particular grading \cite{Ng2005, Plamenevskaya2006,Rasmussen2010}, a count of a collection of gradings \cite{Watson2012}, or a dimension count \cite{KM2011});  and second, the quantity measured is not one that can be extracted from the Jones polynomial --- additional structure in Khovanov homology is essential in each case. The latter points to a clear advantage of Khovanov homology over the Jones polynomial, while the former suggests that further applications might be possible by considering more of the available structure. 

This paper is principally concerned with developing new applications of the graded information in Khovanov homology.

\subsection*{Tangle invariants in Khovanov homology} As with the Jones polynomial, tangle decompositions provide an approach to calculation and an enrichment of structure in Khovanov homology. For example, Bar-Natan's work \cite{Bar-Natan2005} gave rise to a considerable improvements in calculation speed \cite{Bar-Natan2007}. Bar-Natan works in a category of formal complexes of tangles up to homotopy (modulo certain topological relations). On the other hand,  Khovanov defines an algebraic invariant that is more natural in certain settings \cite{Khovanov2002} --- particularly in relation to two-fold branched covers and bordered Floer homology \cite{AGW2011}. There are a range of other generalized tangle invariants in Khovanov homology \cite{APS2006,GW2010,LP2009,Roberts2013-A,Roberts2013-D} and the state of the art is nicely summarized by Roberts \cite{Roberts2013-D}. 

\parpic[r]{\includegraphics[scale=0.5]{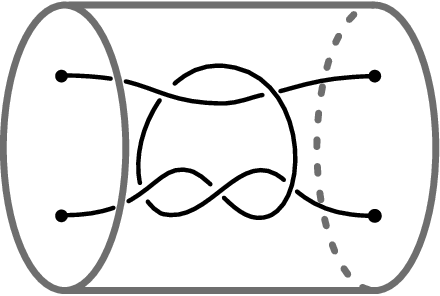}}
We introduce a new tangle invariant in Khovanov homology that is perhaps best aligned with the work of Grigsby and Wehrli \cite{GW2010}. The tangles considered are pairs $T=(B^3,\tau)$, where $\tau$ is a pair of properly embedded disjoint smooth arcs (together with a potentially empty collection of embedded disjoint closed components). These tangles will be endowed with a sutured structure (see Definition \ref{def:suture}, and compare the definitions of \cite[Section 5]{GW2010}), which may be thought of as a partition of the four points of $\partial\tau\subset\partial B^3$ into two pairs of points. Namely, we replace $B^3$ with the product $D^2\times I$ and distinguish the annular subset of the boundary $\partial D^2 \times I$ as the suture.  Equivalence of sutured tangles is up to homeomorphism of the pair $(B^3,\tau)\cong (D^2 \times I,\tau)$ fixing the suture. 

Given a representative $T$ for the homeomorphism class of a sutured tangle, there is a naturally defined link $T(i)$, for any integer $i$, by adding $i$ half-twists and then closing the tangle (as in Figure \ref{fig:closures}).  While these twists do not alter (the homeomorphism class of) the sutured tangle, the links $T(i)$ typically form an infinite family of distinct links. However, the Khovanov homology groups of the $T(i)$ are closely related, owing to the existence of a long exact sequence in Khovanov homology associated with a crossing resolution. In particular, there is a linear map $f_i\co \Khred(T(i+1))\to\Khred(T(i))$ for each integer $i$. Our object of study is the vector space defined by the inverse limit \[\KHT(T)=\varprojlim \Khred(T(i))\] as this yields an invariant of the underlying sutured tangle. It is not immediately apparent how this invariant might be related to other tangle invariants in the literature. While this is an interesting line of inquiry we will leave it for the moment and turn instead to an application.

\subsection*{The symmetry group of a knot} The symmetry group $\Sym(S^3,K)$ of a knot $K$  in $S^3$ is identified with the mapping class group of the knot exterior $M_K=S^3\smallsetminus\nu(K)$ \cite{Kawauchi1996}. 

\labellist
\small
\pinlabel $h$ at 52 4
\endlabellist
\parpic[r]{\includegraphics[scale=0.85]{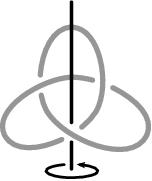}}
A strong inversion on a knot $K$  is an element $h\in\Sym(S^3,K)$ arising from an orientation preserving involution  on $S^3$ that reverses orientation on the knot $K$. The pair $(K,h)$ will be called a strongly invertible knot whenever $h\in\Sym(S^3,K)$ is a strong inversion (this notation follows Sakuma \cite{Sakuma1986}). Notice that, according to the Smith conjecture, the fixed point set of such an involution must be unknotted \cite{Waldhausen1969}. When restricting a strong inversion to $M_K$ we obtain an involution on the knot exterior with one dimensional fixed point set. The quotient of such an involution is a tangle; the arcs of the tangle are the image of the fixed point set in the quotient. Moreover, by choosing equivariant meridional  sutures on $\partial M_K$, the quotient tangle is naturally a sutured tangle $T_{K,h}$ for which the closure $T_{K,h}(\frac{1}{0})$ is the trivial knot. In fact, there is a one-to-one correspondence between pairs $(K,h)$ and sutured tangles $T$ for which $T(\frac{1}{0})$ is the trivial knot. Thus, to any strongly invertible knot $(K,h)$ we may associate a sutured tangle and the invariant $\KHT(T_{K,h})$. 

We will focus on a particular finite dimensional quotient  $\kh(K,h)$ of this inverse limit. This is a $\bZ$-graded vector space; there is a natural secondary relative grading admitting a lift to a $(\bZ\times\odd$)-graded vector space (see Section \ref{sec:bi}). 

Some remarks are in order. If $T=(D^2\times I,\tau)$, then the above construction shows that $M_K$ is the two-fold branched cover  of $D^2\times I$ with branch set $\tau$, denoted $\Br_T$. Moreover, this covering can be chosen so that it respects the sutured structures. It is important to note that $K$ may admit more than one strong inversion and hence it can be the case that $M_K$ may be realised as a two-fold branched cover of $D^2\times I$ in different ways. 

The appropriate notion of equivalence of strongly invertible knots is given by conjugacy in $\Sym(S^3,K)$; see Definition \ref{def:strong-knot}. As such, our invariant is best framed as an invariant of conjugacy classes. For example, if $K$ is hyperbolic it is known that $\Sym(S^3,K)$ is a subgroup of a dihedral group \cite{Riley1979,Sakuma1986} and $K$ admits 0, 1 or 2 strong inversions (up to conjugacy). Furthermore, in the case that there are 2 strong inversions, these must generate a cyclic or free involution \cite{Sakuma1986}. As a result, invariants of strong inversions (particularly, of strongly invertible knots) detect additional structure in the symmetry group. We note that, in general, a given knot admits finitely many strong inversions \cite{Kojima1983}.  

\subsection*{Results and conjectures} An interesting feature of this invariant of strong inversions from Khovanov homology is the following:

\begin{theorem}\label{thm:unknot}
Let $(K,h)$ be a strongly invertible knot. Then $\kh(K,h)=0$ if and only if $K$ is the trivial knot. \end{theorem}

Note that the trivial knot admits a standard strong inversion, and that $(K,h)$ is trivial if and only if $K$ is the trivial knot \cite{Marumoto1977}.  We consider some particular examples as a means of comparing $\kh(K,h)$ with $\Khred(K)$. These establish:

\begin{theorem}\label{thm:summary} (1) There exist distinct knots $K_1$ and $K_2$, each admitting a unique strong inversion $h_1$ and $h_2$ respectively, for which $\Khred(K_1) \cong \Khred(K_2)$ but $\kh(K_1,h_1)\ncong\kh(K_2,h_2)$ as graded vector spaces.\\[4pt]
(2) There exist non-amphicheiral knots $K$, admitting a unique strong inversion $h$, for which $\Khred(K)\cong\Khred(K^*)$ but $\kh(K,h)\ncong\kh(K^*,h)$ as graded vector spaces.
\end{theorem}

In fact we show more: Of all knots with 10 or fewer crossing for which the Jones polynomial and the signature (in combination) fail to detect non-amphicheirality, there is an involution present and $\kh$ detects non-amphicheirality; see Theorem \ref{thm:ten}. 

Sakuma introduces and studies a similar invariant $\eta(K,h)$ \cite{Sakuma1986}. This is an Alexander-like polynomial invariant that, like $\kh(K,h)$, vanishes for the trivial knot. Unlike $\kh(K,h)$, Sakuma's invariant  also vanishes for a range of non-trivial strongly invertible knots, including any amphicheiral $(K,h)$ for which $h$ is unique up to conjugacy \cite[Proposition 3.4]{Sakuma1986}. In this context it is also worth mentioning the work of Couture defining a Khovanov-like homology associated with links of divides \cite{Couture2009}. This gives rise to a homological invariant of strongly invertible knots $(K,h)$, though it is not clear how this is related to $\kh(K,h)$ (or if the two are related at all); see Remark \ref{rmk:Couture}.

Our work points to some conjectures about the behaviour of the vector space $\kh(K,h)$ and, most notably, its relationship with Heegaard Floer homology. In particular, there is evidence suggesting that the family of strongly invertible L-space knots --- knots admitting a Dehn surgery with simplest-possible Heegaard Floer homology --- might be characterised by way of Khovanov homology by appealing to $\kh$; see Conjecture \ref{con:L-space}.

\subsection*{Organization} Background on Khovanov homology is collected in Section \ref{sec:Kh} with particular attention paid to our grading conventions and their relationship to other conventions in the literature; this is summarized in Figure \ref{fig:gradings}. The invariants of sutured tangles and of strongly invertible knots are defined in Section \ref{sec:inv}; the invariant of strongly invertible knots $\kh$ is the main focus of the remainder of the paper. In Section \ref{sec:properties} some basic properties of $\kh$ are established including the non-vanishing result (Theorem \ref{thm:unknot}). In Section \ref{sec:examples} we give some preliminary examples. This includes properties of $\kh$ for torus knots (Theorem \ref{thm:torus}) and highlights the invariant's ability to distinguish strong inversions; compare Question \ref{qst:seperate}. Section \ref{sec:amph} considers the problem of obstructing amphicheirality and establishes  Theorem  \ref{thm:summary}. Section \ref{sec:conjectures} presents three conjectures.  The invariant $\kh$ is graded, but also comes with a natural relative bi-grading that can be useful in calculations. The paper concludes with a construction of a lift of the latter to an absolute bi-grading in Section \ref{sec:bi}. 
 
\section{Khovanov homology} \label{sec:Kh}

Khovanov's categorification of the Jones polynomial gives rise to a (co)homological invariant of oriented links in the three-sphere with the Jones polynomial arising as a graded Euler characteristic \cite{Khovanov2000}. For the purpose of this paper it will be sufficient to work with the reduced Khovanov homology $\Khred(L)$  \cite{Khovanov2003} taking coefficients in the field $\bF=\bZ/2\bZ$ and giving rise to a $(\bZ\times\frac{1}{2}\bZ)$-graded vector space. Letting $u$ denote the integer (homological) grading and $q$ denote the half-integer (quantum) grading, the invariant satisfies $\Khred(U)\cong\bF$ supported in grading $(u,q)=(0,0)$ (where $U$ denotes the trivial knot) and, more generally, \[V_L(t)=\sum_{q\in\bZ}b_qt^q\] where each coefficient $b_q=\chi_u\big(\Khred_q(L)\big)=\sum_{u\in\bZ}(-1)^u\dim\Khred^u_q(L)$ is the Euler characteristic in a fixed quantum grading. The symmetry in the Jones polynomial $V_{L^*}(t)=V_L(t^{-1})$, where $L^*$ denotes the mirror image of $L$, is realised in the bi-grading of Khovanov homology as $\Khred^u_q(L^*)\cong\Khred^{-u}_{-q}(L)$.
 
Note that the quantum grading used here is half the grading considered elsewhere (compare \cite{Khovanov2000}, for example) and results in half-integer powers for links with an even number of components. 

There is a third natural grading to consider: Setting $\delta=u-q$ records diagonals of slope 1 in the $(u,q)$-plane and gives rise to a $\frac{1}{2}\bZ$-grading on $\Khred(L)$. This may be relaxed to a relative $\bZ$-grading (compare \cite{Watson2012}). It is an absolute $\bZ$-grading for knots and we have that \[\det(K) = |\chi_\delta\Khred(K)|\]
where $\chi_\delta\big(\Khred(K)\big)=\sum_{\delta\in\bZ}(-1)^\delta\dim\Khred^\delta(K)$ (ignoring the quantum grading).

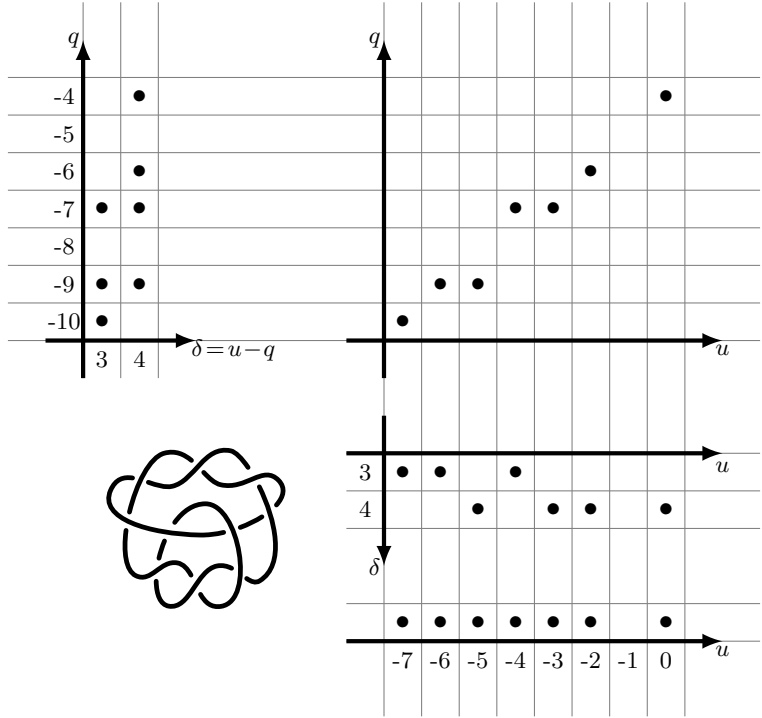
\begin{figure}
\begin{tikzpicture}[>=latex] 

\draw [gray] (-5,3.5) -- (5,3.5);
\draw [gray] (-5,3) -- (5,3);
\draw [gray] (-5,2.5) -- (5,2.5);
\draw [gray] (-5,2) -- (5,2);
\draw [gray] (-5,1.5) -- (5,1.5);
\draw [gray] (-5,1) -- (5,1);
\draw [gray] (-5,0.5) -- (5,0.5);
\draw [gray] (-5,0) -- (5,0);

\draw [gray] (4,-5) -- (4,4.5);
\draw [gray] (3.5,-5) -- (3.5,4.5);
\draw [gray] (3,-5) -- (3,4.5);
\draw [gray] (2.5,-5) -- (2.5,4.5);
\draw [gray] (2,-5) -- (2,4.5);
\draw [gray] (1.5,-5) -- (1.5,4.5);
\draw [gray] (1,-5) -- (1,4.5);
\draw [gray] (0.5,-5) -- (0.5,4.5);
\draw [gray] (0,-5) -- (0,4.5);

\draw [gray] (-3,-0.5) -- (-3,4.5);
\draw [gray] (-3.5,-0.5) -- (-3.5,4.5);
\draw [gray] (-4,-0.5) -- (-4,4.5);

\draw [gray] (-0.5,-1.5) -- (5,-1.5);
\draw [gray] (-0.5,-2) -- (5,-2);
\draw [gray] (-0.5,-2.5) -- (5,-2.5);

\draw [gray] (-0.5,-3.5) -- (5,-3.5);
\draw [gray] (-0.5,-4) -- (5,-4);


\node at (4.5,-0.125) {\footnotesize{$u$}};
\node at (-0.125,4) {\footnotesize{$q$}};

\draw[ultra thick,->] (0,-0.5) -- (0,4);
\draw[ultra thick,->] (-0.5,0) -- (4.5,0);

\node at (0.25,-4.25) {\footnotesize{-$7$}};
\node at (0.75,-4.25) {\footnotesize{-$6$}};
\node at (1.25,-4.25) {\footnotesize{-$5$}};
\node at (1.75,-4.25) {\footnotesize{-$4$}};
\node at (2.25,-4.25) {\footnotesize{-$3$}};
\node at (2.75,-4.25) {\footnotesize{-$2$}};
\node at (3.25,-4.25) {\footnotesize{-$1$}};
\node at (3.75,-4.25) {\footnotesize{$0$}};

\node at (-4.25,0.25) {\footnotesize{-$10$}};
\node at (-4.25,0.75) {\footnotesize{-$9$}};
\node at (-4.25,1.25) {\footnotesize{-$8$}};
\node at (-4.25,1.75) {\footnotesize{-$7$}};
\node at (-4.25,2.25) {\footnotesize{-$6$}};
\node at (-4.25,2.75) {\footnotesize{-$5$}};
\node at (-4.25,3.25) {\footnotesize{-$4$}};

\node at (0.25,0.25) {$\bu$};
\node at (0.75,0.75) {$\bu$};
\node at (1.75,1.75) {$\bu$};

\node at (1.25,0.75) {$\bu$};
\node at (2.25,1.75) {$\bu$};
\node at (2.75,2.25) {$\bu$};
\node at (3.75,3.25) {$\bu$};


\node at (-2,-0.125) {\footnotesize{$\delta\!=\!u\!-\!q$}};
\node at (-4.125,4) {\footnotesize{$q$}};

\node at (-3.25,-0.25) {\footnotesize{$4$}};
\node at (-3.75,-0.25) {\footnotesize{$3$}};

\draw[ultra thick,->] (-4,-0.5) -- (-4,4);
\draw[ultra thick,->] (-4.5,0) -- (-2.5,0);

\node at (-3.75,0.25) {$\bu$};
\node at (-3.75,0.75) {$\bu$};
\node at (-3.75,1.75) {$\bu$};

\node at (-3.25,0.75) {$\bu$};
\node at (-3.25,1.75) {$\bu$};
\node at (-3.25,2.25) {$\bu$};
\node at (-3.25,3.25) {$\bu$};


\draw[ultra thick,->] (-0.5,-1.5) -- (4.5,-1.5);
\node at (4.5,-1.675) {\footnotesize{$u$}};
\draw[ultra thick,->] (0,-1) -- (0,-3);
\node at (-0.125,-3.0125) {\footnotesize{$\delta$}};

\node at (-0.25,-2.25) {\footnotesize{$4$}};
\node at (-0.25,-1.75) {\footnotesize{$3$}};

\node at (0.25,-1.75) {$\bu$};
\node at (0.75,-1.75) {$\bu$};
\node at (1.75,-1.75) {$\bu$};

\node at (1.25,-2.25) {$\bu$};
\node at (2.25,-2.25) {$\bu$};
\node at (2.75,-2.25) {$\bu$};
\node at (3.75,-2.25) {$\bu$};


\draw[ultra thick,->] (-0.5,-4) -- (4.5,-4);
\node at (4.5,-4.125) {\footnotesize{$u$}};

\node at (0.25,-3.75) {$\bu$};
\node at (0.75,-3.75) {$\bu$};
\node at (1.75,-3.75) {$\bu$};

\node at (1.25,-3.75) {$\bu$};
\node at (2.25,-3.75) {$\bu$};
\node at (2.75,-3.75) {$\bu$};
\node at (3.75,-3.75) {$\bu$};

\node at (-2.5,-2.5) {\includegraphics[scale=0.75]{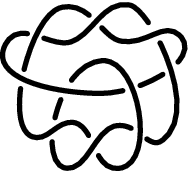}};
\end{tikzpicture}
\caption{A gradings glossary: $u$ (cohomological), $q$ (quantum) and $\delta$ (diagonal) gradings on the vector space $\Khred(K)$. For the purpose of illustration we have considered the torus knot $K\simeq 10_{124}$. Each $\bu$ denotes a copy of the vector space $\bF$.  Notice that $V_K(t)=t^{-4}+t^{-6}-t^{-10}$ in this case. For reference, the conventions in the upper left correspond to those of \cite{Watson2012,Watson2013}.}\label{fig:gradings}\end{figure}

These conventions are consistent with \cite{MO2007,Watson2012,Watson2013} and are summarized for a particular example in Figure \ref{fig:gradings}. Two different gradings on Khovanov homology will be used in this paper:
\begin{itemize}
\item[(1)] A finite dimensional $\bZ$-graded vector space $\Khred(L)=\Khred^u(L)$ by considering the homological grading $u$ (and ignoring both $q$ and $\delta$);  \\
\item[(2)] a finite dimensional $(\bZ\times\frac{1}{2}\bZ)$-graded vector space $\Khred(L)=\Khred^{u,\delta}(L)$ by considering the homological grading $u$ and the diagonal grading $\delta$. We will generally relax the half-integer grading to an integer grading at the expense of passing from an absolute grading to a relative grading in the second factor.   \\
\end{itemize}

With these conventions in place, we review the long exact sequence associated with a crossing resolution. Let $[\cdot,\cdot]$ be an operator on the bi-grading satisfying \[\Khred^{u,\delta}(L)[i,j]\cong\Khred^{{u-i},{\delta-j}}(L)\] and, given an orientation on (a fixed diagram of)  $L$, let $n_-(L)$ record the number of negative crossings according to a right-hand rule. Then given a distinguished positive crossing $\rightcross$ in a link diagram fix $c= n_-(\one)-n_-(\positive)$ for some choice of orientation on the affected strands of the new link associated with the resolution $\one$ at the crossing --- note that this is the resolution that does not inherit the orientation of the original link. With the abuse of notation $\Khred(L)=\Khred(\positive)$, $n_-(L)=n_-(\positive)$ (and so forth) understood, we have a long exact sequence 
\[\begin{tikzpicture}[>=latex] 
\matrix (m) [matrix of math nodes, row sep=1em,column sep=1em]
{ \Khred(\positive) & & \Khred(\zero)[0,-\half] \\
& \Khred(\one)[c+1,-\half c] & \\ };
\path[->,font=\scriptsize]
(m-1-1) edge node[auto] {$ f^+ $} (m-1-3)
(m-2-2.north west) edge (m-1-1.south east)
(m-1-3.south west) edge node[auto] {$ \partial $} (m-2-2.north east);
\end{tikzpicture}\]
The connecting homomorphism $\partial$ is graded of bi-degree $(1,1)$, that is, this map raises both  $u$- and  $
\delta$-grading by 1. Note that the long exact sequence is particularly well behaved with respect to the $\bZ$-grading:
\[\begin{tikzpicture}[>=latex] 
\matrix (m) [matrix of math nodes, row sep=1em,column sep=1em]
{ \Khred^u(\positive) & & \Khred^u(\zero) \\
& \Khred^{u-c-1}(\one) & \\ };
\path[->,font=\scriptsize]
(m-1-1) edge node[auto] {$ f^+ $} (m-1-3)
(m-2-2.north west) edge (m-1-1.south east)
(m-1-3.south west) edge node[auto] {$ \partial $} (m-2-2.north east);
\end{tikzpicture}\]
That is, this exact {\em triangle} encodes the long exact sequence 
\[\begin{tikzpicture}[>=latex] 
\matrix (m) [matrix of math nodes, row sep=1em,column sep=2em]
{ \cdots & \Khred^{u-c-1}(\one) & \Khred^u(\rightcross) & \Khred^u(\zero) & \Khred^{u-c}(\one) & \cdots\\ };
\path[->,font=\scriptsize]
(m-1-1) edge (m-1-2)
(m-1-2) edge (m-1-3)
(m-1-3) edge node[auto] {$ f^+ $} (m-1-4)
(m-1-4) edge node[auto] {$ \partial $} (m-1-5)
(m-1-5) edge (m-1-6);
\end{tikzpicture}\]
and the map $f^+\co \Khred(\positive) \to \Khred(\zero)$ preserves the cohomological grading.  

\section{Invariants from inverse limits}\label{sec:inv}

\subsection{An invariant of sutured tangles} \label{sec:tangles}
A tangle $T$ is the homeomorphism class of a pair $(B^3,\tau)$ where $B^3$ is a three-ball and $\tau$ is a pair of properly embedded arcs (together with a potentially empty collection of embedded circles). Consider the identification $B^3\cong D^2\times I$.
  
\begin{definition}\label{def:suture} A sutured tangle is a pair $(D^2\times I,\tau)$ where the four endpoints $\partial \tau$ are divided into two pairs confined to $D^2\times\{0\}$ and $D^2\times\{1\}$ respectively. The suture is the annulus $\partial D^2\times I$; equivalence of sutured tangles is up to homeomorphism of the pair $(D^2\times I,\tau)$ fixing the suture pointwise. A sutured tangle is called braid-like if the arcs $\tau$ admit an orientation that is inward at $D^2\times \{0\}$ and outward at $D^2\times \{1\}$. 
\end{definition}

An example is illustrated in Figure \ref{fig:suture}. All tangles considered in this work will be sutured. 

\begin{figure}[h]
\includegraphics[scale=0.75]{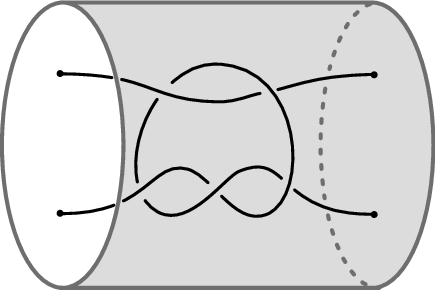}\qquad\qquad\qquad
\includegraphics[scale=0.75]{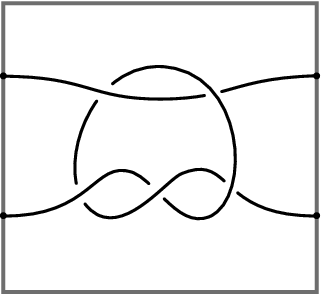}
\caption{An example of a sutured tangle with the suture $\partial D^2\times I$ shaded (left); and projection to $I\times I$ illustrating the convention for planar representations of sutured tangles (right). Note that this example is braid-like in the sense of Definition \ref{def:suture}.}\label{fig:suture}\end{figure}

Having fixed a representative for a sutured tangle $T$, there are two natural links obtained in the closure: $T(\frac{1}{0})$ joins the endpoints in $D^2\times \{0\}$ and $D^2\times \{1\}$ respectively without adding any new crossings; $T(0)$ joins each endpoint in  $D^2\times \{0\}$ to an endpoint in $D^2\times \{1\}$ without adding new crossings (see Figure \ref{fig:closures}).

\begin{figure}[h]
\labellist
\pinlabel $\underbrace{\phantom{aaaaaaaaaaaaia}}_n$ at 360 26
\large 
\pinlabel $T$ at 41 49
\pinlabel $T$ at 160 49
\pinlabel $T$ at 279 49
\pinlabel $\cdots$ at 375 49
\endlabellist
\includegraphics[scale=0.75]{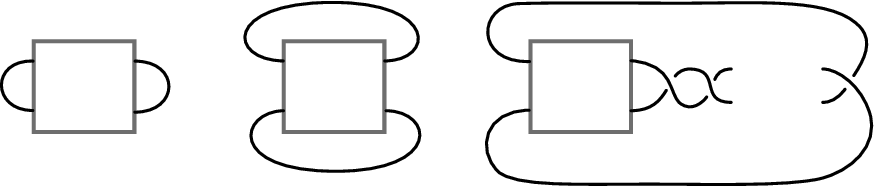}
\caption{Links $T(\frac{1}{0})$, $T(0)$ and $T(n)$ (from left-to-right) obtained via the closure of a fixed representative of a sutured tangle $T$. }\label{fig:closures}\end{figure}
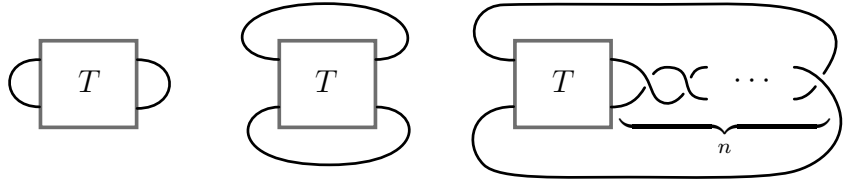

More generally, note that the homeomorphism class of $T$ is not altered by adding horizontal twists, that is, a homeomorphism exchanging the pair of points  $\partial\tau|_{D^2\times\{1\}}\subset D^2\times\{1\}$. With the convention that the crossing $\positive$ is represented by $+1$, the obvious one-parameter family of tangle representatives gives rise to an infinite family of links $T(n)$ in the closure. Precisely, if $T^n$ is the representative obtained from $T$ by adding $n$ half-twists, then we have $T(n)=T^n(0)$ (see Figure \ref{fig:closures}). Rational tangle attachments other than these horizontal twists will not, in general, preserve the suture despite the fact that the equivalence class of the underlying (unsutured) tangle is preserved (see \cite{Watson2012}, for example, for details on this more standard notion of tangle equivalence). 

Fix a representative of a sutured tangle $T$ and, with the above conventions for closures in place, define $A_i=\Khred(T(i))$ for all  $i\in\bZ$. Then there is a natural inverse system provided by the maps $f_i\co A_{i+1}\to A_i$ in the long exact sequence. These are not necessarily graded maps since the resolved crossing need not be positive for a general tangle: At present we are distinguishing between $f_i$ and $f_i^+$ depending on compatibility of orientations at the resolved crossing. Note however that $f_i$ may always be regarded as a relatively graded map between the bi-graded vector spaces $A_{i+1}$ and $A_i$. Define \[\KHT(T) = \varprojlim A_i,\] the inverse limit of $\{A_i,f_i\}$ (see Weibel \cite{Weibel1994}, for example). We have by construction that:

\begin{proposition}
The vector space $\KHT(T)$ is an invariant of the sutured tangle $T$, up to isomorphism. Moreover, if $T$ is braid-like then $\KHT(T)$ is naturally $\bZ$-graded.  
\end{proposition}

\begin{proof}
The proposition follows immediately from the definitions owing to the fact that the pair $\{A_i,f_i\}$ is an invariant of $T$ up to reindexing. However the grading in the second statement deserves a few words. 

If $T$ is a representative of a braid-like sutured tangle then the link $T(i)$, with Khovanov invariant $A_i$, inherits an orientation from the braid-like orientation on the properly embedded arcs of $T$. As a result, with this orientation fixed, the long exact sequence may be expressed as
\[\begin{tikzpicture}[>=latex] 
\matrix (m) [matrix of math nodes, row sep=1em,column sep=1em]
{ A_{i+1} & & A_i\\
& B & \\ };
\path[->,font=\scriptsize]
(m-1-1) edge node[auto] {$ f^+_i $} (m-1-3)
(m-2-2.north west) edge (m-1-1.south east)
(m-1-3.south west) edge  (m-2-2.north east);
\end{tikzpicture}\]

where $B=\Khred(T(\frac{1}{0}))[c_T+i+1]=\Khred^{u-c_T-i-1}(T(\frac{1}{0}))$. The integer $c_T$ counts the negative crossings in $T$ when the orientation on one of the arcs of $\tau$ is reversed so that \[\textstyle c=n_-(T(\frac{1}{0}))+i-n_-(T(i+1))=c_T+i.\]  Now the directed system $(A_i,f^+_i)$ is graded in the sense that each $f_i^+$ is a $\bZ$-graded map between $\bZ$-graded vector spaces. As a result the inverse limit $\KHT(T)$ inherits this grading as claimed. 
\end{proof}

\subsection{An invariant of strongly invertible knots}\label{sec:strong}

\begin{definition} A strong inversion $h$ on a knot $K$ is the isotopy class of an orientation preserving homeomorphism of $S^3$ that reverses orientation on the knot $K$. \end{definition}

Note that the fixed point set of $h$ is necessarily unknotted as a consequence of the Smith conjecture \cite{Waldhausen1969}. The involution $h$ is an element of the symmetry group of the knot, denoted $\Sym(S^3,K)$, which is identified with  the mapping class group of the knot exterior $M_K\cong S^3\smallsetminus \nu(K)$. Properties of this group are summarized by Kawauchi \cite[Chapter 10]{Kawauchi1996}; strong inversions in particular are considered by Sakuma \cite{Sakuma1986}. While the symmetry group of a knot may be trivial, and in particular, a given knot might not admit a strong inversion, these are relatively natural objects. For example:

\begin{theorem}[Schreier \cite{Schreier1924}, see {\cite[Exercise 10.6.4]{Kawauchi1996}} or {\cite[Proposition 3.1 (1)]{Sakuma1986}}]\label{thm:Schreier} If $K$ is a torus knot then $\Sym(S^3,K)\cong\bZ/2\bZ$ is generated by a unique strong inversion on $K$. $\hfill\ensuremath\Box$ \end{theorem}

\begin{theorem}[Thurston, see {\cite[Page 124]{Riley1979}} and {\cite[Proposition 3.1 (2)]{Sakuma1986}}] \label{thm:dihedral}If $K$ is a hyperbolic knot then $\Sym(S^3,K)$ is a subgroup of a dihedral group. In particular, $K$ admits 0, 1 or 2 strong inversions up to conjugacy, and $K$ admits 2 strong inversions if and only if $K$ admits a free or cyclic involution.$\hfill\ensuremath\Box$\end{theorem}

More generally, any given knot admits finitely many strong inversions \cite{Kojima1983}. 

\begin{definition}\label{def:strong-knot}
A strongly invertible knot is a pair $(K,h)$ where $K$ is a knot in $S^3$ and $h$ is a strong inversion on $K$. Equivalence of strongly invertible knots $(K,h)$ and $(K',h')$ is up to orientation preserving homeomorphism $f\co S^3\to S^3$ satisfying $f(K)=K'$ and $fhf^{-1}=h'$. In particular, a strongly invertible knot corresponds to the conjugacy class of a strong inversion in $\Sym(S^3,K)$. 
\end{definition}

If $(K,h)$ is a strongly invertible knot then the knot exterior $M_K$ admits an involution $h|_{M_K}$ with one dimensional fixed-point set meeting the boundary torus in four points. Moreover, the quotient  of $M_K$ by the involution $h|_{M_K}$ is necessarily homeomorphic to $B^3$ (see \cite[Proposition 3.5]{Watson2012}, for example), and the image of the fixed-point set in the quotient is a pair of properly embedded arcs. 

Choose a pair of disjoint annuli in $\partial M_K$, equivariant  with respect to $h$, with meridional cores. Then the quotient of $M_K$ (as an equivariantly sutured manifold) is a sutured tangle denoted $T_{K,h}$; see Figure \ref{fig:si}. Notice that, by construction, the closure $T_{K,h}(\frac{1}{0})$ provides a branch set for the trivial surgery on $K$ and is therefore the trivial knot.

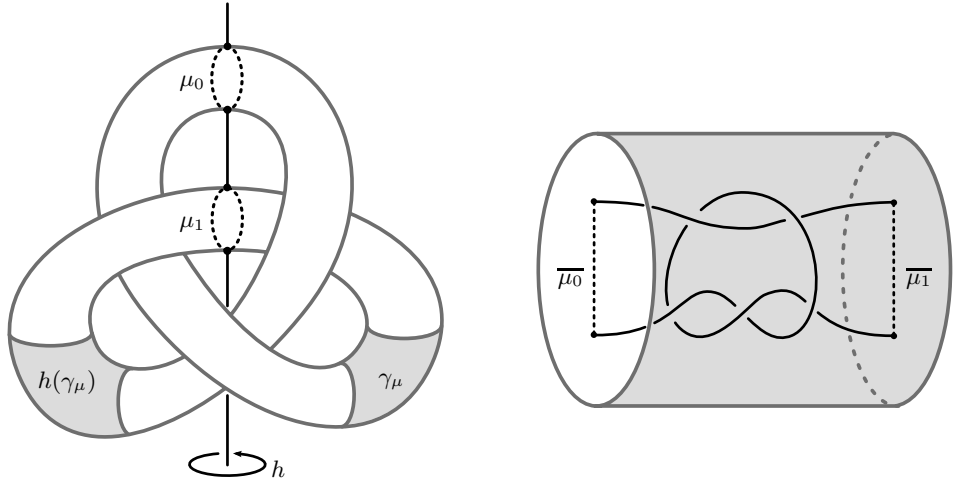
\begin{figure}
\labellist\footnotesize
\pinlabel $h$ at 136 4
\pinlabel $h(\gamma_\mu)$ at 30 48 
\pinlabel $\gamma_\mu$ at 193 48 
\pinlabel $\mu_0$ at 93 198
\pinlabel $\mu_1$ at 93 127

\pinlabel $\overline{\mu_0}$ at 283 100
\pinlabel $\overline{\mu_1}$ at 458 100
\endlabellist
\includegraphics[scale=0.75]{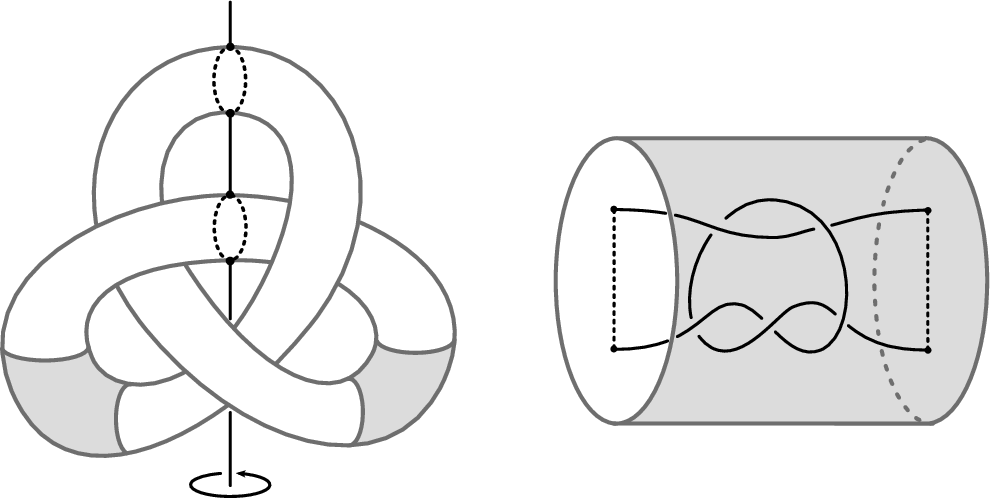}
\caption{The trefoil is strongly invertible. On the left, the involution on the complement of the trefoil is illustrated. Note that this symmetry exchanges the annular sutures $\gamma_\mu$ and $h(\gamma_\mu)$ in the boundary  while fixing the meridians $\mu_0$ and $\mu_1$. On the right, the resulting quotient is a sutured tangle where each meridian descends to an arc $\operatorname{im}(\mu_i)=\overline{\mu_i}\in D^2\times\{i\}$ for $i=0,1$. This representative for the quotient tangle shown is compatible with the framing $6\mu+\lambda$ (in terms of the preferred generators); more on this quotient may be found in \cite[Section 3]{Watson2012} and \cite[Section 2]{Watson2013}.}\label{fig:si}\end{figure}

\begin{proposition}\label{prp:tangle}
There is a one-to-one correspondence between strongly invertible knots $(K,h)$ and sutured tangles satisfying the additional property that $T(\frac{1}{0})$ is the trivial knot. 
\end{proposition}

\labellist
\large \pinlabel $T$ at 52 22.5
\small \pinlabel $a$ at 2 53
\endlabellist
\parpic[r]{\includegraphics[scale=0.75]{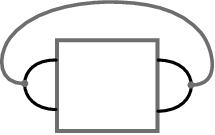}}
{\it Proof.} This is immediate from the discussion above. To reverse the construction, notice that the two-fold branched cover of the trivial knot $T(\frac{1}{0})$ is $S^3$, and the lift of an unknotted arc $a$ meeting the two lobes of the closure defining $T(\frac{1}{0})$ is a strongly invertible knot in $S^3$. $\hfill\ensuremath\Box$

Note that, if $T_{K,h}= (D^2\times I,\tau)$ is the tangle associated with a strongly invertible knot $(K,h)$, then the above construction realises the knot exterior $M_K$ as the two-fold branched cover of $D^2\times I$, with branch set $\tau$, denoted $\Br_{T_{K,h}}$. In particular, a given distinct (conjugacy classes of) strong inversions $h,h'$ on $K$ we get distinct strongly invertible knots $(K,h)$, $(K,h')$  (in the sense of Definition \ref{def:strong-knot}) and the knot exterior may be realised as a two-fold branched cover in two distinct ways. 

Moreover, the Dehn surgery $S^3_n(K)$ on a strongly invertible knot $(K,h)$ may be realised as the two-fold branched cover of $S^3$ with branch set $T_{K,h}(n)$ for a suitable choice of representative (this is the preferred representative of \cite[Section 3.4]{Watson2012}). Note that this is a generalization/reformulation of the Montesinos trick \cite{Montesinos1976-trick}.

It is an immediate consequence of the property that $T(\frac{1}{0})$ is unknotted that the sutured tangle $T$ is braid-like. In view of Proposition \ref{prp:tangle}, given a strongly invertible knot $(K,h)$ we can associate the $\bZ$-graded invariant $\KHT(T_{K,h})$.

Let $\boldA=\KHT(T_{K,h})$ and recall that $x\in\boldA$ may be identified with a sequence $\{x_j\}_{j\in\bZ}$ such that $f_j(x_{j+1})=x_j$, where $x_j\in\Khred(T(j))$, for all $j\in\bZ$. 

\begin{definition}\label{def:strong-inv} 
 Given a strongly invertible knot $(K,h)$ consider the subspace $\boldK\subset\boldA$ consisting of sequences satisfying the additional condition that $x_j=0$ for $j\ll0$.  Denote by $\kh(K,h)$ the vector space $\boldA/\boldK$.
\end{definition}

 \begin{proposition}\label{prp:subspace}
 The vector space $\kh(K,h)$ is a  finite-dimensional $\bZ$-graded invariant of the strongly invertible knot $(K,h)$, up to isomorphism. 
  \end{proposition}

\begin{proof}
This can be seen from the short exact sequence of vector spaces \[
\begin{tikzpicture}[>=latex] 
\matrix (m) [matrix of math nodes, row sep=1.5em,column sep=1.5em]
{ 0&\boldK&\boldA& \kh(K,h) & 0 \\
 };
\path[->,font=\scriptsize]
(m-1-1) edge (m-1-2) 
(m-1-2) edge (m-1-3) 
(m-1-3) edge (m-1-4) 
(m-1-4) edge (m-1-5) 
;
\end{tikzpicture}
\] which, owing to the fact that the inclusion is graded, may be decomposed according to the $\bZ$-grading. That is, $\kh(K,h)\cong\bigoplus_{u\in\bZ}\kh^u(K,h)$ where $\kh^u(K,h)\cong \boldA^u/\boldK^u$ is the $u^{\text{th}}$ graded piece of $\kh(K,h)$ having decomposed the inclusion $\boldK^u\into\boldA^u$ according to the $\bZ$-grading. 

Setting $A_j=\Khred(T_{K,h}(j))$, so that $\boldA=\varprojlim A_j$ as in the definition of $\KHT(T_{K,h})$, any choice of splitting  gives rise to a (non-canonical) inclusion $\sigma\co\kh(K,h)\into\boldA$. 
Recall that the universal property for the inverse limit is summarized in the present case by the commutative diagram
\[\begin{tikzpicture}[>=latex] 
\matrix (m) [matrix of math nodes, row sep=1.5em,column sep=2em]
{ & \kh(K,h) & \\
 & \boldA & \\
A_{j+1}&  & A_j \\ };
\path[->,font=\scriptsize]
(m-1-2) edge[right hook->] (m-2-2)  
(m-2-2) edge node[above] {$ \pi_{j+1} $} (m-3-1)
(m-2-2) edge node[above] {$ \pi_j  $} (m-3-3)
(m-1-2) edge[bend right] node[left] {$ \iota_{j+1} $} (m-3-1)
(m-1-2) edge[bend left] node[right] {$ \iota_j  $} (m-3-3)
(m-3-1) edge node[above] {$f_j$} (m-3-3);
\node at (0.35,0.555) {\scriptsize$\sigma$};
\node at (-3.2,-1.1) {$\cdots$};\draw[->] (-2.9,-1.1) -- (-2.4,-1.1);
\node at (3.2,-1.1) {$\cdots$};\draw[->] (2.4,-1.1) -- (2.9,-1.1);
\end{tikzpicture}\]
resulting in a family of maps $\iota_j=\pi_j\circ\sigma$ for $j\in\bZ$. Note that $\ker(\iota_j)=\ker(\sigma)=0$ since $\sigma(y)_j\in A_j$ must be non-zero, for all $j\in\bZ$, for any given non-zero element $y\in\kh(K,h)$. As a result this construction gives injections $\iota_j\co\kh(K,h)\into A_j$ for all $j\in\bZ$ and, since $A_j$ is finite dimensional, $\kh(K,h)$ must be finite dimensional as well.
\end{proof} 

We reiterate that distinguishing strong inversions on a particular knot gives additional information about the symmetry group. In particular, if $\kh(K,h)\ncong\kh(K,h')$ then we have identified distinct conjugacy classes of involutions in $\Sym(S^3,K)$. Moreover, if $K$ is hyperbolic and $\kh(K,h)\ncong\kh(K,h')$, then there must be a free or cyclic involution on the knot complement and hence a third element of order two in $\Sym(S^3,K)$ (see Theorem \ref{thm:dihedral}).  

\section{Properties}\label{sec:properties}

\subsection{Behaviour under mirror image}\label{sub:mirror} A key property of the invariant $\kh$ is inherited from Khovanov homology. 

\begin{proposition}\label{prp:mirror}
Let $(K,h)$ be a strongly invertible knot and denote by $(K,h)^*$ the strongly invertible mirror, obtained by reversing orientation on $S^3$. Then $\kh^u(K,h)^*\cong\kh^{-u}(K,h)$ as $\bZ$-graded vector spaces. 
\end{proposition}

Note that the strongly invertible mirror need not fix the conjugacy class of $h\in\Sym(S^3,K)$ in the case that the underlying knots are amphicheiral (that is, when $K^*\simeq K$); compare \cite[Proposition 4.3]{Sakuma1986}

\begin{proof}[Proof of Proposition \ref{prp:mirror}] From the construction of the tangle $T_{K,h}$ associated with a strongly invertible knot $(K,h)$ we have that $T^*_{K,h}=T_{(K,h)^*}$ is the tangle associated with the strongly invertible mirror $(K,h)^*$. From the forgoing discussion, any choice of graded section $\sigma\co\kh(K,h)\into\KHT(T_{K,h})$ gives rise to a family of graded inclusion maps $\iota_j=\pi_j\circ\sigma$ for $j\in\bZ$. Now, for any $j\in\bZ$, we have 
\[\begin{tikzpicture}[>=latex] 
\matrix (m) [matrix of math nodes, row sep=1.5em,column sep=2.5em]
{ \kh^u(K,h)^*  &   \Khred^u(T_{K,h}^*(-j)) \\
\kh^{-u}(K,h) &     \Khred^{-u}(T_{K,h}(j))\\};
\path[->,font=\scriptsize]
(m-1-1) edge[right hook->] node[above] {$\iota_{-j}$} (m-1-2)
(m-2-1) edge[right hook->] node[above] {$\iota_{j}$} (m-2-2)
(m-1-2) edge[<->] node[right] {$\cong$}(m-2-2);
\end{tikzpicture}\]
by applying the behaviour of Khovanov homology for mirrors since $(T_{K,h}(j))^*\simeq T^*_{K,h}(-j)$. Composing with the isomorphism gives inclusions establishing that $\kh^u(K,h)^*\subseteq\kh^{-u}(K,h)$ and $\kh^{-u}(K,h)\subseteq\kh^u(K,h)^*$. As a result we have the identification  $\kh^u(K,h)^*\cong \kh^{-u}(K,h)$ as claimed.
\end{proof}

This may be obtained, alternatively, from the more general observation that \[\KHT^{-u}(T_{K,h})\cong\KHT^{u}(T^*_{K,h}).\] Given that $\Khred^{-u}(T_{K,h}(j))\cong\Khred^u(T_{K,h}^*(-j))$ we leave the reader to check that the relevant linear maps $f_{-j}$ and $f^*_{j-1}$ exchanged correspond to projections and inclusions, respectively, of complexes at the chain level.  In particular, there is an analogous inverse system associated with resolutions of negative crossings arising from the long exact sequence for a negative crossing. 

\subsection{Notions of stability} A key feature of Khovanov homology, leading ultimately to the computability of $\kh(K,h)$, is that the vector space $\Khred(T(n))$ stabilises, in a suitable sense, for sufficiently large $n$. This is made precise in the following statement (compare \cite[Lemma 4.10]{Watson2012}, taking note of the change in grading convention).

\begin{lemma}\label{lem:stability}
Fix a representative $T=T_{K,h}$ for the sutured tangle associated with a strongly invertible knot $(K,h)$. Let $X$ be the one dimensional bi-graded vector space $\bF^{(c_T,\phalf(1-c_T))}\cong\Khred(T(\frac{1}{0}))[c_T,\frac{1}{2}(1-c_T)]$ where $c_T$ is the difference in negative crossings between the braid-like and non-braid-like orientation on the arcs of $T$. Then, up to an overall $-\frac{n}{2}$ in the $\delta$-grading, \[\Khred(T(n))\cong H_*\Big(\Khred(T(0)) \overset{D}{\to} \bigoplus_{i=0}^nX[i,0]\Big)\] obtained from an iterated mapping cone construction where $D$ is of bi-degree $(1,1)$. 
\end{lemma}
\begin{proof}
We fix the constant $c_T=n_-(T(\frac{1}{0}))-n_-(T(0))$ throughout, and let $n>0$ so that $c=n_-(T(\frac{1}{0}))-n_-(T(n))=c_T+n$. Now considering iterated applications of the long exact sequence (and minding gradings!), $\Khred(T(n))$ may be computed in terms of $\Khred(T(0))$:
\[
\begin{tikzpicture}[>=latex] 
\matrix (m) [matrix of math nodes, row sep=1em,column sep=2em]
{    \Khred(T(n)) &\\
 \Khred(T(n-1))[0,-\half]  &\Khred(T(\frac{1}{0}))[c_T+n,-\half(1-c_T-n)] \\ 
 \Khred(T(n-2))[0,-1] &  \Khred(T(\frac{1}{0}))[c_T+n-1,-\half(1-c_T-n)]  \\
   \vdots & \vdots\\
      \Khred(T(0))[0,-\frac{n}{2}] & \Khred(T(\frac{1}{0}))[c_T,-\half(1-c_T-n)] \\} ;
\path[->,font=\scriptsize]
(m-1-1) edge node[auto] {$ f_{n-1} $} (m-2-1)
(m-2-1) edge node[auto] {$ f_{n-2} $} (m-3-1)
(m-3-1) edge node[auto] {$ f_{n-3} $} (m-4-1)
(m-4-1) edge node[auto] {$ f_{0} $} (m-5-1)

(m-2-2) edge[<-] node[auto] {$ \partial $} (m-2-1)
(m-3-2) edge[<-] node[auto] {$ \partial$} (m-3-1)
(m-5-2) edge[<-] node[auto] {$ \partial$} (m-5-1)

;
\end{tikzpicture}
\] 
Recall that the connecting homomorphisms $\partial$ are of bi-degree $(1,1)$ and, in particular, raise the $\delta$-grading by one. As a result, since the occurrences of $\Khred(T(\frac{1}{0}))$ are in a fixed $\delta$-grading, this iterative process does not induce any maps between the $\Khred(T(\frac{1}{0}))$. Hence $\Khred(T(n))$ may be computed from a complex (or, mapping cone; see Weibel \cite{Weibel1994}, for example) of the form
\[
\begin{tikzpicture}[>=latex] 
\matrix (m) [matrix of math nodes, row sep=0.5em,column sep=2em]
{  
 & X[n,-\frac{n}{2}]  \\ 
 & X[n-1,-\frac{n}{2}]  \\
   & \vdots \\
      \Khred(T(0))[0,-\frac{n}{2}] &X[0,-\frac{n}{2}] \\} ;
\path[->]
(m-4-1) edge[bend left]   (m-1-2)
(m-4-1) edge[bend left]  (m-2-2)
(m-4-1) edge[bend left] node[above] {$\vdots$}  (m-4-2)
;
\end{tikzpicture}
\] 
where each of the depicted maps is induced from the connecting homomorphism. The total map $D$ therefore raises the bi-grading by $(1,1)$ and the homology of $D$ gives the result as claimed. 
\end{proof}

There are two immediate and important consequences of this observation.

\begin{corollary}[See {\cite[Lemma 4.14]{Watson2012}}]\label{cor:stability}
If $n\gg0$ then $\Khred(T(n+1))\cong\Khred(T(n))\oplus\bF$. $\hfill\ensuremath\Box$
\end{corollary}

\begin{corollary}\label{cor:inj/surj}
The map $f_i$ is surjective for all $i\gg 0$ and injective for all $i\ll 0$.  $\hfill\ensuremath\Box$
\end{corollary}

Corollary \ref{cor:inj/surj} is an essential observation: It ensures computability of $\kh(K,h)$ and $\KHT(T_{K,h})$. Note that {\em sufficiently large/small} in this context depends, in general, on the choice of representative for the sutured tangle. On the other hand, varying the choice of representative can be a useful trick for computing $\kh(K,h)$ (see Section \ref{sec:examples}).

\subsection{Detecting the trivial knot}The object $\kh(K,h)$ bears some similarities with a polynomial invariant $\eta(K,h)$ of strongly invertible knots considered by Sakuma \cite{Sakuma1986}. In particular, Sakuma proves that $\eta$  is zero for the trivial knot. However, Sakuma's invariant must also vanish for $(K,h)$ if $K$ is amphicheiral and $h$ is a unique strong inversion on $K$ (up to conjugacy) \cite[Proposition 3.4 (1)]{Sakuma1986}. A stronger statement holds for $\kh(K,h)$ (compare \cite[Section 4.6]{Watson2012}). 

\begin{named}{Theorem \ref{thm:unknot}}Let $(K,h)$ be a strongly invertible knot. Then $\kh(K,h)=0$ if and only if $K$ is the trivial knot. \end{named}

\parpic[r]{\includegraphics[scale=0.75]{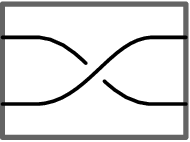}}
{\it Proof.} First suppose that $K$ is the trivial knot, and notice that $T_{K,h}(n)$ may be identified with the $(2,n-1)$-torus link. This choice of representative for the sutured tangle is illustrated on the right; notice that $T(1)$ is the two-component trivial link and $T(2)$ and $T(0)$ are both trivial knots. The result now follows from direct calculation (compare \cite[Section 6.2]{Khovanov2000}). To see this, observe that if some composition $f=f_i\circ\cdots\circ f_j$ of the inverse system is the zero map, then $\kh(K,h)$ must vanish. Indeed, in this situation any sequence $\{x_j\}$ for which $f_{j}(x_{j+1})=x_j$ and $x_j\in A_j$ must also satisfy $x_j=0$ for $j\ll0$. Hence $\boldA\cong\boldK$ (in the notation of Definition \ref{def:strong-inv}). 

We simply observe that $f_1\circ f_0=0$ in the present setting. In detail, the long exact sequence defining $f_1$ is 
\[\begin{tikzpicture}[>=latex] 
\matrix (m) [matrix of math nodes, row sep=1em,column sep=1em]
{ \Khred(T(2)) & & \Khred(T(1))[0,-\half] \\
& \Khred(T({\frac{1}{0}}))[1,0] & \\ };
\path[->,font=\scriptsize]
(m-1-1) edge node[auto] {$ f_1 $} (m-1-3)
(m-2-2.north west) edge (m-1-1.south east)
(m-1-3.south west) edge node[auto] {$ \partial $} (m-2-2.north east);
\end{tikzpicture}\]
which simplifies to  
\[\begin{tikzpicture}[>=latex] 
\matrix (m) [matrix of math nodes, row sep=1em,column sep=1em]
{ \bF^{(0,0)} & & \bF^{(0,0)}\oplus\bF^{(0,\miniminus1)}\\
& \bF^{(1,0)} & \\ };
\path[->,font=\scriptsize]
(m-1-1) edge node[auto] {$ f_1 $} (m-1-3)
(m-2-2) edge node[auto] {$0$} (m-1-1)
(m-1-3) edge  (m-2-2);
\end{tikzpicture}\]
so that, in particular, if $x_0$ generates $\bF^{(0,0)}$ then $f_1(x_0)=(x_0,0)\in \bF^{(0,0)}\oplus\bF^{(0,\miniminus1)}$. Similarly, to define $f_0$ we have 
\[\begin{tikzpicture}[>=latex] 
\matrix (m) [matrix of math nodes, row sep=1em,column sep=1em]
{ \Khred(T(1)) & & \Khred(T(0))[0,-\half] \\
& \Khred(T(\frac{1}{0}))[0,\half] & \\ };
\path[->,font=\scriptsize]
(m-1-1) edge node[auto] {$ f_0 $} (m-1-3)
(m-2-2.north west) edge (m-1-1.south east)
(m-1-3.south west) edge node[auto] {$ \partial $} (m-2-2.north east);
\end{tikzpicture}\]
which simplifies to give 
\[\begin{tikzpicture}[>=latex] 
\matrix (m) [matrix of math nodes, row sep=1em,column sep=1em]
{ \bF^{(0,\phalf)}\oplus\bF^{(0,\mhalf)} & & \bF^{(0,\mhalf)}\\
& \bF^{(0,\phalf)} & \\ };
\path[->,font=\scriptsize]
(m-1-1) edge node[auto] {$ f_0 $} (m-1-3)
(m-2-2) edge  (m-1-1)
(m-1-3) edge node[auto] {$0$} (m-2-2);
\end{tikzpicture}\]
where $f_0(x_{\phalf},x_{\mhalf})=(0,x_{\mhalf})$ given $(x_{\phalf},x_{\mhalf})\in \bF^{(0,\phalf)}\oplus\bF^{(0,\mhalf)}$. Composing these two maps  yields $(f_0\circ f_1)(x_0) = f_0(x_0,0)=0$ as claimed. 

The converse depends on a relationship with Heegaard Floer homology. Let $(K,h)$ be a strongly invertible knot and suppose that $\kh(K,h)\cong0$. For an appropriately chosen representative of $T=T_{K,h}$ we have that $S^3_n(K)$ is the two-fold branched cover  $\Br_{T(n)}$. 

Now let $n\gg 0$ so that by Corollary \ref{cor:stability} we have $\Khred(T(n+1))\cong\Khred(T(n))\oplus\bF$. Then given a graded section $\sigma\co \kh(K,h)\into \KHT(T_{K,h})$ we may write \begin{align*}\Khred(T(n+1))&\cong\kh(K,h)\oplus\bF^{m+1}\\  \Khred(T(n))&\cong\kh(K,h)\oplus\bF^{m}\end{align*} for some $m\ge0$. Applying Lemma \ref{lem:stability}, $\bF^{m}$ and $\bF^{m+1}$ are supported in a single (relative) $\delta$-grading. Indeed, these subspaces cannot be present in $\Khred(T(n))$ when $n\ll0$, so they must cancel in the associated iterated mapping cone. Note that $m\le n$ is determined by those $f_{i>0}$ that are surjective (as in Corollary \ref{cor:inj/surj}).
By choosing $n\gg 0$ the connecting homomorphism vanishes for grading reasons and we obtain a short exact sequence
\[\begin{tikzpicture}[>=latex] 
\matrix (m) [matrix of math nodes, row sep=1.5em,column sep=1.5em]
{ 0&\bF&{ \kh(K,h)\oplus\bF^{m+1}}& \kh(K,h)\oplus\bF^{m} & 0 \\
 };
\path[->,font=\scriptsize]
(m-1-1) edge (m-1-2) 
(m-1-2) edge (m-1-3) 
(m-1-3) edge node[above] {$ f_n $} (m-1-4) 
(m-1-4) edge (m-1-5) 
;
\end{tikzpicture}\]
(where the grading shifts have been suppressed). But $\kh(K,h)$ vanishes by hypotheses, so $\Khred(T(n))\cong \bF^m$ is supported in a single $\delta$-grading. It follows that \[m=\det(T(n))=|H_1(\Br_{T(n)};\bZ)|=|H_1(S_n(K);\bZ)|\] and $m=n$. 

Given a link $L$ there is a spectral sequence starting from $\Khred(L)$ and converging to $\HFhat(-\Br_L)$  (here, $-\Br_L$ denotes the two-fold branched cover $\Br_L$ with orientation reversed) \cite{OSz2005-branch}. In the present setting \[n=|H_1(S^3_n(K);\bZ)|\le\dim\HFhat(S^3_n(K))\le\dim\Khred(T(n)) = n\] so that $\dim\HFhat(S^3_n(K))=|H_1(S_n^3(K);\bZ)|$ and hence $S^3_n(K)$ is an L-space for all $n\gg0$. As a result, $g(K)=\tau(K)$, where $g(K)$ is the genus of the knot and $\tau(K)$ is the Ozsv\'ath-Szab\'o concordance invariant \cite[Proposition 3.3]{OSz2005-lens}. 

A nearly identical argument for $n\ll0$ shows that $S^3_n(K)$ is an L-space for all $|n|\gg 0$.  Thus $S^3_{-n}(K)\cong S^3_n(K^*)$ is an L-space for sufficiently large $n$ and hence $g(K)=g(K^*)=\tau(K^*)=-\tau(K)$ \cite[Lemma 3.3]{OSz2003}. It follows that  $\tau(K)=g(K)=0$ and hence $K$ is the trivial knot.$\hfill\ensuremath\Box$

\subsection{An aside on cabling} While not every knot $K$ is strongly invertible, it is always the case that $D(K)=K\# K$ is a strongly invertible knot (with canonical strong inversion that switches the components, which we will suppress from the notation). As a result, the relatively graded vector space resulting from the composite $\kh(D(-))$ is an invariant of knots in $S^3$ (compare \cite{HW2010}). This invariant detects the trivial knot as a consequence of Theorem \ref{thm:unknot}, and is closely related to work of  Grigsby and Wehrli. Indeed, we obtain an alternate proof of the following:

\begin{theorem}[Grigsby-Wehrli \cite{GW2010}, Hedden \cite{Hedden2009-unknot}]\label{thm:cable}
The Khovanov homology of the two-cable of a knot detects the trivial knot.
\end{theorem} 

\begin{proof}[Sketch of Proof] Let $T$ be the tangle associated with the quotient of $D(K)$ with representative chosen so that  $S^3_0(D(K))\cong \Br_{T(0)}$. We appeal to two immediate facts. First, that $D(K)$ is the trivial knot if and only if $K$ is the trivial knot, and second, that $T(0)$ is the (untwisted) two-cable of $K$ denoted $\sC_2K$. We need to show  that if $\Khred(\sC_2K)\cong \bF^2$ then $\sC_2K$ is the two-component trivial link (and hence $K$ is trivial).  

Note that since $\kh(D(K))$ injects into $\Khred(T(0))\cong\Khred(\sC_2K)$ we may write $\Khred(\sC_2K)\cong\kh(D(K))\oplus\bF^n$ with the summand $\bF^n$ supported in a single $\delta$-grading (Lemma \ref{lem:stability}).  Since $\kh(D(K))$ vanishes only for the trivial knot (Theorem \ref{thm:unknot}) we may assume that $\kh(D(K))$ is non-zero. Furthermore, $0=\det(T(0))=|\chi_\delta\Khred(\sC_2K)|$ so if $n\ge 2$ we are done. This leaves two cases to consider: either $n=0$ and $\dim\kh(D(K))=2$ or $n=1$ and $\dim\kh(D(K))=1$. Notice that, in either case, $\Khred(\sC_2(K))\cong\bF^2$. 

Now applying the Ozsv\'ath-Szab\'o spectral sequence for the two-fold branched cover, together with the non-vanishing of $\HFhat(S^3_0(D(K)))$, we have that $\HFhat(S^3_0(D(K)))\cong\bF^2$ \cite{OSz2005-branch}. Now  $S^3_0(D(K))$ is prime \cite[Corollary 4.5]{Scharlemann1990} so by a result of  Hedden and Ni $S^3_0(D(K))$ must be $S^2\times S^1$ or  0-surgery on a trefoil  \cite[Theorem 1.1]{HN2010}. The latter can be ruled out by hand (see the first example of Section \ref{sec:examples}; compare \cite{HW2010}) hence $\sC_2K$ must be the two-component trivial link.  
\end{proof}

This proof is not appreciably different from those already in the literature and represents essentially a reorganizing of the data on the $E_2$-page of a spectral sequence. However, it illustrates an interesting point: The proof could be shortened and made more internal to Khovanov homology were it known that $\dim\kh(K,h)>2$ for $K$ non-trivial. This is worth advertising as something stronger appears to be true; see Conjecture \ref{con:structure}.

\section{Examples}
\label{sec:examples}

We now turn to calculations of $\kh(K,h)$ for some explicit examples. While this invariant is defined as a $\bZ$-graded vector space, in practice (as seen in establishing some of the properties in Section \ref{sec:properties}) it is useful to make use of the secondary (relative) $\bZ$-grading from $\delta=u-q$, which is described in Section \ref{sec:Kh}.

\subsection{Torus knots} We begin by giving a relatively detailed calculation for the invariant associated with the right-hand trefoil (our running example through the paper which we denote here by $K$). The strong inversion is shown in Figure \ref{fig:si} together with the quotient tangle. Note that the representative depicted satisfies $S^3_6(K)\cong\Br_{T(0)}$ since the connect sum in the branch set identifies the reducible surgery on this torus knot (see Moser \cite{Moser1971}). As a result, $T(-5)$ may be identified with the (negative) $(3,5)$-torus knot; see Figure \ref{fig:gradings}. 

It will be convenient to fix the preferred representative $T^\circ$ for this tangle satisfying $S^3_n(K)\cong\Br_{T^\circ(n)}$. With this choice we have $T^\circ(1)$ is the knot $10_{124}$ of Figure \ref{fig:gradings} and $T^\circ(6)$ is the connected sum as above. We focus on the portion 
\[\begin{tikzpicture}[>=latex] 
\matrix (m) [matrix of math nodes, row sep=1em,column sep=1em]
{\phantom{\cdot} &\Khred(T^\circ(6)) &  \Khred(T^\circ(5)) &  \Khred(T^\circ(4)) &  \Khred(T^\circ(3)) &  \Khred(T^\circ(2)) & \Khred(T^\circ(1)) & \phantom{\cdot}  \\
 };
\path[->,font=\scriptsize]
(m-1-1) edge (m-1-2)
(m-1-2) edge node[auto] {$ f_{5} $} (m-1-3)
(m-1-3) edge node[auto] {$ f_{4} $} (m-1-4)
(m-1-4) edge node[auto] {$ f_{3} $} (m-1-5)
(m-1-5) edge node[auto] {$ f_{2} $} (m-1-6)
(m-1-6) edge node[auto] {$ f_{1} $} (m-1-7)
(m-1-7) edge  (m-1-8);
\end{tikzpicture}\]
of the inverse system.  The key observation is that each of these maps, and indeed all of the $f_i$, are determined by $\Khred(T^\circ(1))$ and $\Khred(T^\circ(6))$ together with Lemma \ref{lem:stability}. Namely, in the notation of Lemma \ref{lem:stability} we have that 
\[\Khred(T^\circ(6)) \cong H_*\Big(\Khred(T^\circ(1)) \overset{D}{\to}\bF^5\Big)\]
and, since $T^\circ(6)$ is a connect sum of two-bridge knots, $\Khred(T^\circ(6))\cong\bF^6$ must be supported in a single $\delta$-grading (that is, alternating links are {\em thin}, see Lee \cite{Lee2005}). The precise (graded) form of this invariant is easily calculated and is given in Figure \ref{fig:trefoil-calc}.  On the other hand, from Figure \ref{fig:gradings} we have that $\Khred(T^\circ(1))\cong\bF^3\oplus\bF^4$ supported in two adjacent $\delta$-gradings. As a result, we have that 
\[\Khred(T^\circ(6))\cong H_*\Big(\Khred(T^\circ(1)) \overset{D}{\to} \bigoplus_{i=0}^4\bF^{(u=i-6)}\Big)\] 
up to an overall shift in the $\delta$-grading. 
\begin{figure}[tt]
\begin{tikzpicture}[>=latex] 

\fill [lightgray] (4,7.25)-- (3.5,7.25) -- (3.5,6)--(4,6);

\fill [lightgray] (3,7.25)-- (3.5,7.25) -- (3.5,4.75)--(3,4.75);

\fill [lightgray] (3,7.25)-- (2.5,7.25) -- (2.5,3.5)--(3,3.5);

\fill [lightgray] (2,7.25)-- (2.5,7.25) -- (2.5,2.25)--(2,2.25);

\fill [lightgray] (1,7.25)-- (1.5,7.25) -- (1.5,0)--(1,0);

\draw [gray] (4,-5) -- (4,7.25);
\draw [gray] (3.5,-5) -- (3.5,7.25);
\draw [gray] (3,-5) -- (3,7.25);
\draw [gray] (2.5,-5) -- (2.5,7.25);
\draw [gray] (2,-5) -- (2,7.25);
\draw [gray] (1.5,-5) -- (1.5,7.25);
\draw [gray] (1,-5) -- (1,7.25);
\draw [gray] (0.5,-5) -- (0.5,7.25);
\draw [gray] (0,-5) -- (0,7.25);
\draw [gray] (-0.5,-5) -- (-0.5,7.25);

\draw [gray] (-0.5,6.5) -- (5,6.5);
\draw [gray] (-0.5,6) -- (5,6);

\draw [gray] (-0.5,4.75) -- (5,4.75);
\draw [gray] (-0.5,5.25) -- (5,5.25);

\draw [gray] (-0.5,3.5) -- (5,3.5);
\draw [gray] (-0.5,4) -- (5,4);

\draw [gray] (-0.5,2.25) -- (5,2.25);
\draw [gray] (-0.5,2.75) -- (5,2.75);

\draw [gray] (-0.5,1) -- (5,1);
\draw [gray] (-0.5,1.5) -- (5,1.5);

\draw [gray] (-0.5,-0.25) -- (5,-0.25);
\draw [gray] (-0.5,0.25) -- (5,0.25);

\draw [gray] (-0.5,-1.5) -- (5,-1.5);
\draw [gray] (-0.5,-1) -- (5,-1);

\draw [gray] (-0.5,-2.25) -- (5,-2.25);
\draw [gray] (-0.5,-2.75) -- (5,-2.75);

\draw [gray] (-0.5,-3.5) -- (5,-3.5);
\draw [gray] (-0.5,-4) -- (5,-4);

\node at (-0.25,-4.75) {\footnotesize{-$7$}};
\node at (0.25,-4.75) {\footnotesize{-$6$}};
\node at (0.75,-4.75) {\footnotesize{-$5$}};
\node at (1.25,-4.75) {\footnotesize{-$4$}};
\node at (1.75,-4.75) {\footnotesize{-$3$}};
\node at (2.25,-4.75) {\footnotesize{-$2$}};
\node at (2.75,-4.75) {\footnotesize{-$1$}};
\node at (3.25,-4.75) {\footnotesize{$0$}};
\node at (3.75,-4.75) {\footnotesize{$1$}};


\node at (4.75, -3.75) {$A_1$};

\node at (-0.25,-3.75) {$\circ$};
\node at (0.25,-3.75) {$\circ$};
\node at (1.25,-3.75) {$\circ$};

\node at (0.75,-3.75) {$\bu$};
\node at (1.75,-3.75) {$\bu$};
\node at (2.25,-3.75) {$\bu$};
\node at (3.25,-3.75) {$\bu$};

\node at (4.75, -2.5) {$A_2$};

\node at (0.25,-2.5) {$\circ$};
\node at (1.25,-2.5) {$\circ$};

\node at (0.75,-2.5) {$\bu$};
\node at (1.75,-2.5) {$\bu$};
\node at (2.25,-2.5) {$\bu$};
\node at (3.25,-2.5) {$\bu$};

\node at (4.75, -1.25) {$A_3$};

\node at (1.25,-1.25) {$\circ$};

\node at (0.75,-1.25) {$\bu$};
\node at (1.75,-1.25) {$\bu$};
\node at (2.25,-1.25) {$\bu$};
\node at (3.25,-1.25) {$\bu$};

\node at (4.75, 0) {$A_4$};

\node at (1.25,-0.12) {$\circ$};
\node at (1.25,0.12) {$\bu$};

\node at (0.75,0) {$\bu$};
\node at (1.75,0) {$\bu$};
\node at (2.25,0) {$\bu$};
\node at (3.25,0) {$\bu$};

\node at (4.75, 1.25) {$A_5$};

\node at (1.25,1.25) {$\bu$};

\node at (0.75,1.25) {$\bu$};
\node at (1.75,1.25) {$\bu$};
\node at (2.25,1.25) {$\bu$};
\node at (3.25,1.25) {$\bu$};

\node at (4.75, 2.5) {$A_6$};

\node at (1.25,2.5) {$\bu$};

\node at (0.75,2.5) {$\bu$};
\node at (1.75,2.5) {$\bu$};
\node at (2.25,2.6) {$\bu$};\node at (2.25,2.4) {$\bu$};
\node at (3.25,2.5) {$\bu$};

\node at (4.75, 3.75) {$A_7$};

\node at (1.25,3.75) {$\bu$};

\node at (0.75,3.75) {$\bu$};
\node at (1.75,3.75) {$\bu$};
\node at (2.25,3.65) {$\bu$};\node at (2.25,3.85) {$\bu$};
\node at (3.25,3.75) {$\bu$};
\node at (2.75,3.75) {$\bu$};

\node at (4.75, 5) {$A_8$};

\node at (1.25,5) {$\bu$};

\node at (0.75,5) {$\bu$};
\node at (1.75,5) {$\bu$};
\node at (2.25,4.9) {$\bu$};\node at (2.25,5.1) {$\bu$};
\node at (2.75,5) {$\bu$};
\node at (3.25,4.9) {$\bu$};\node at (3.25,5.1) {$\bu$};

\node at (4.75, 6.25) {$A_9$};

\node at (1.25,6.25) {$\bu$};

\node at (0.75,6.25) {$\bu$};
\node at (1.75,6.25) {$\bu$};
\node at (2.25,6.15) {$\bu$};\node at (2.25,6.35) {$\bu$};
\node at (2.75,6.25) {$\bu$};
\node at (3.25,6.15) {$\bu$};\node at (3.25,6.35) {$\bu$};
\node at (3.75,6.25) {$\bu$};

\draw[thick,->>] (4.75,6) -- (4.75,5.25);
\draw[thick,->>] (4.75,4.75) -- (4.75,4);
\draw[thick,->>] (4.75,3.5) -- (4.75,2.75);
\draw[thick,->>] (4.75,2.25) -- (4.75,1.5);
\draw[thick,->] (4.75,1) -- (4.75,0.25);
\draw[thick,->] (4.75,-.25) -- (4.75,-1);
\draw[thick,left hook->] (4.75,-1.5) -- (4.75,-2.25);
\draw[thick,left hook->] (4.75,-2.75) -- (4.75,-3.5);

\node at (4.75,-4.25) {$\vdots$}; 
\node at (4.75,7) {$\vdots$}; 
\node at (0.75,7) {$\vdots$};
\node at (1.25,7) {$\vdots$};
\node at (1.75,7) {$\vdots$};
\node at (2.25,7) {$\vdots$};
\node at (2.75,7) {$\vdots$};
\node at (3.25,7) {$\vdots$};
\node at (3.75,7) {$\vdots$};
\node at (-3.5,-2.5) {\includegraphics[scale=0.75]{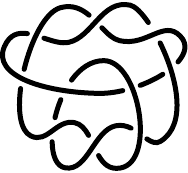}}; \draw [very thick, black,->] (-2.4,-3.1) .. controls (-2,-3.5) and (-1.25,-3.75) .. (-0.5,-3.75);

\node at (-4,1) {\includegraphics[scale=0.7]{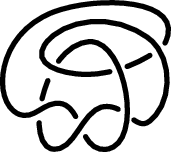}};  \draw [very thick, black,->] (-3,0.5) .. controls (-2.5,0) and (-1.25,1.25) .. (-0.5,1.25);
\node at (-2.5,2.75) {\includegraphics[scale=0.6]{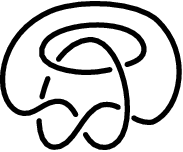}};  \draw [very thick, black,->] (-1.5,2.65) .. controls (-1,2.65) and (-1.25,2.5) .. (-0.5,2.5);
\node at (-4,4) {\includegraphics[scale=0.5]{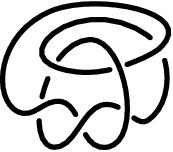}}; \draw [very thick, black,->] (-3.15,4.25) .. controls (-2,4.25) and (-1.25,3.75) .. (-0.5,3.75);
\end{tikzpicture}
\caption{Calculations for the unique strong inversion on $K$, the right-hand trefoil: The representative of the associated quotient tangle $T^\circ$ is chosen so that $T^\circ(1)\simeq 10_{124}$ (compare Figure \ref{fig:gradings}) and $T^\circ(6)$ is a connect sum of a trefoil and a Hopf link. There are two relative $\delta$-gradings distinguished here by the conventions that $\circ$ generates a copy of $\bF$ in grading $\delta$ and $\bu$ generates a copy of $\bF$ in grading $\delta+1$; recall that the connecting homomorphisms raise both $u$- and $\delta$-grading by one. With the convention that $A_i=\Khred(T^\circ(i))$ the sequences contributing to $\boldK\subset\boldA$ have been shaded so that, for example,  $\boldA^0\cong \bF^2$ while $\kh^0(K)\cong\bF$. }
\label{fig:trefoil-calc}
\end{figure}
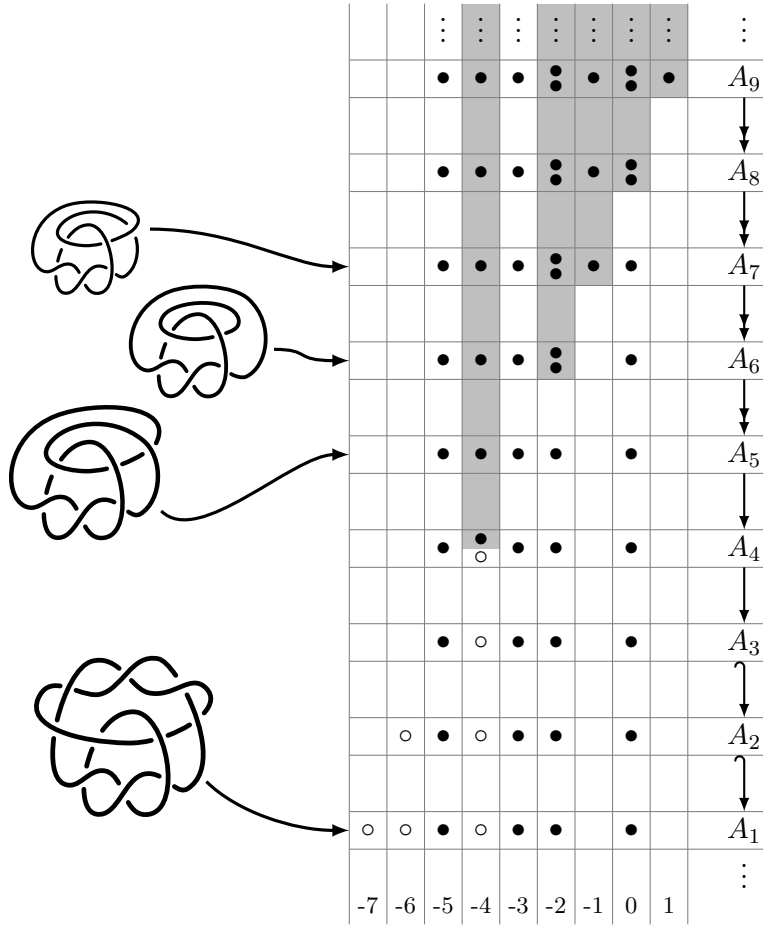
The differential must cancel the off-diagonal $\bF^3$ to yield a thin knot; all other differentials are necessarily trivial. This calculation is summarized in Figure \ref{fig:trefoil-calc}. In particular, setting $A_i=\Khred(T^\circ(i))$ the long exact sequence splits in the cases
\[\begin{tikzpicture}[>=latex] 
\matrix (m) [matrix of math nodes, row sep=0.5em,column sep=1em]
{0 &\bF & A_{i+1}&  A_i &  0 & &  {\text{when}\ } i\ge5 \\
0 & A_{i+1}&  A_i & \bF & 0 && {\text{when}\ } i\le2\\
 };
\path[->,font=\scriptsize]
(m-1-1) edge (m-1-2)
(m-1-2) edge  (m-1-3)
(m-1-3) edge node[auto] {$ f_{i} $} (m-1-4)
(m-1-4) edge (m-1-5)
(m-2-1) edge (m-2-2)
(m-2-2) edge node[auto] {$ f_{i} $}  (m-2-3)
(m-2-3) edge  (m-2-4)
(m-2-4) edge (m-2-5)
;
\end{tikzpicture} \]
for grading reasons and hence $f_i$ is surjective for $i\ge5$ and $f_i$ is injective for $i\le2$. 

From the foregoing we calculate  
\[\KHT^u(T^\circ)\cong
\begin{cases}
\bF & u\ge1\\
\bF^2 & u=0\\ 
\bF & u = -1\\
\bF^2 & u=-2\\ 
\bF & u = -5,-4,-3 \\
0 & u\le-6
\end{cases}
\] and 
\[\kh^u(K)\cong
\begin{cases}
\bF & u=-5,-3,-2,0\\
0 & {\text{otherwise}}\\
\end{cases}
\]
where the unique strong inversion on $K$ has been suppressed from the notation. Alternatively, given a graded section $\sigma\co\kh(K)\into\KHT(T^\circ)$ we have 
\[\KHT(T^\circ)\cong\kh(K)\oplus u^{-4}\cdot W\]
 where $W$ is the subspace of the graded vector space $\bF[u]$ consisting of $\sum a_iu^i$ for which $a_1=0$. 

We will record the invariant by
\[
\begin{tikzpicture}[scale=0.85]

\draw[step=.5,black] (0,0) grid (3,.5);

\node at (-0.75,0.25) {$\kh(K)$};

\node at (0.25,.25) {$1$}; 

\node at (1.25,.25) {$1$}; 
\node at (1.75,.25) {$1$}; 

\node at (2.75,.25) {$1$}; 

\node at (0.25,-0.25) {\footnotesize{-$5$}}; 
\node at (0.75,-0.25) {\footnotesize{-$4$}}; 
\node at (1.25,-0.25) {\footnotesize{-$3$}}; 
\node at (1.75,-0.25) {\footnotesize{-$2$}}; 
\node at (2.25,-0.25) {\footnotesize{-$1$}}; 
\node at (2.75,-0.25) {\footnotesize{$0$}}; 

\end{tikzpicture}\]
as extracted from Figure \ref{fig:trefoil-calc}. Our convention (here, and in the examples to follow) is that the dimension of the vector space in each $u$-grading is recorded (with blank entries indicating dimension 0); and the $u$-grading (labelled along the bottom) is read from left-to-right, following the conventions in Figure \ref{fig:gradings}. 

\begin{remark}\label{rmk:Couture}Couture has defined a Khovanov-like invariant for signed divides which may be regarded as an invariant of strongly invertible knots \cite{Couture2009}. Consulting \cite[Section 3.6]{Couture2009} we see that Couture's invariant for the trefoil has dimension 6 and is supported in positive gradings, suggesting that our invariant is not an alternate formulation of Couture's. (In fact, the difference is perhaps more pronounced for the unknot where Couture's invariant has dimension 2 and $\kh$ vanishes; see Theorem \ref{thm:unknot}.) While a relationship between the two would be interesting, such a relationship seems unlikely: both invariants are extracted from auxiliary objects associated with a strong inversion however Couture defines an apparently new chain complex while $\kh$ appeals to stable/limiting behaviour of the long exact sequence.  \end{remark}

This trick of appealing to surgeries on torus knots may be applied more generally. We know, for example, that:

\begin{theorem}\label{thm:torus}
For any torus knot $K_{p,q}$ the invariant $\kh(K_{p,q})$ is thin, in the sense that the vector space is supported in a single (relative) $\delta$-grading. Moreover the dimension of   $\kh(K_{p,q})$ is bounded above by $|pq|-1$. 
\end{theorem}

\begin{proof} Up to taking mirrors it suffices to consider the case $p,q>0$. Fix the representative $T^\circ$ for the quotient tangle of $K_{p,q}$ satisfying $S^3_n(K_{p,q})\cong\Br_{T^\circ(n)}$.  Then by a result of Moser we have that $S^3_{pq-1}(K_{p,q})$ is a lens space \cite{Moser1971} so that the branch set $T^\circ(pq-1)$ must be a two-bridge knot by work of Hodgson and Rubinstein \cite{HR1985}.  Now Lee's results establish that $\Khred(T(pq-1))$ is supported in a single $\delta$-grading \cite{Lee2005}. As a result  \[pq-1=|H_1(S^3_{pq-1}(K_{p,q});\bZ)|=\det(T(pq-1)) = |\chi_\delta\Khred(T^\circ(pq-1))| = \dim\Khred(T^\circ(pq-1))\] and both statements follow on observing that $\iota_{pq-1}\co\kh(K_{p,q}) \to \Khred(T^\circ(pq-1))$ is injective, where $\iota_{pq-1}=\pi_{pq-1}\circ\sigma$, for any choice of graded section $\sigma\co\kh(K_{p,q})\into\KHT(T^\circ)$.
\end{proof}

As a result, the same procedure described for the trefoil may be applied to determine $\kh(K_{p,q})$ for any integers $p,q$ (again, omitting the unique strong inversion from the notation). For example, the (positive) torus knots $5_1=K_{2,5}$ and $8_{19}=K_{3,4}$ yield 

\[\begin{tikzpicture}[scale=0.85]

\draw[step=.5,black] (0,0) grid (5,1.5);

\node at (-0.85,1.25) {$\kh(5_1)$};

\node at (0.25,1.25) {$1$}; 
\node at (0.75,1.25) {$1$}; 
\node at (1.25,1.25) {$1$}; 
\node at (1.75,1.25) {$2$}; 
\node at (2.25,1.25) {$1$}; 
\node at (2.75,1.25) {$1$}; 
\node at (3.25,1.25) {$1$}; 
 
\node at (-0.85,0.25) {$\kh(8_{19})$};

\node at (1.75,.25) {$1$}; 
\node at (2.25,.25) {$1$}; 
\node at (2.75,.25) {$1$}; 
\node at (3.25,.25) {$2$}; 
\node at (3.75,.25) {$1$}; 
\node at (4.25,.25) {$1$}; 
\node at (4.75,.25) {$1$}; 

\node at (0.25,-0.25) {\footnotesize{-$2$}}; 
\node at (0.75,-0.25) {\footnotesize{-$1$}}; 
\node at (1.25,-0.25) {\footnotesize{$0$}}; 
\node at (1.75,-0.25) {\footnotesize{$1$}}; 
\node at (2.25,-0.25) {\footnotesize{$2$}}; 
\node at (2.75,-0.25) {\footnotesize{$3$}}; 
\node at (3.25,-0.25) {\footnotesize{$4$}}; 
\node at (3.75,-0.25) {\footnotesize{$5$}}; 
\node at (4.25,-0.25) {\footnotesize{$6$}}; 
\node at (4.75,-0.25) {\footnotesize{$7$}}; 

\draw [white, fill=white] (-0.015,0.515) rectangle (5.015,0.985);

\end{tikzpicture}\]

In both cases the invariant is supported in a single $\delta$-grading in agreement with Theorem \ref{thm:torus}. Notice that these examples are not distinguished as ungraded or relatively graded vector spaces, establishing  the (absolute) $u$-grading as an essential part of the invariant. 

It is interesting to compare this calculation with the behaviour of knot Floer homology \cite{OSz2004-knot, Rasmussen2003} for these examples: One can verify that $\dim\HFK(5_1)=\dim\HFK(8_{19})=5$ but that $\HFK(5_1)\ncong\HFK(8_{19})$ as graded vector spaces.

\subsection{Distinguishing strong inversions}\label{sec:dist} For all remaining examples we will state the result of our calculation, while specifying the strong inversion and the associated quotient tangle so that the reader can reproduce our work if desired. Our interest in this section will be on distinguishing strongly invertible knots $(K,h_1)$ and $(K,h_2)$, or, separating conjugacy classes in $\Sym(S^3,K)$.     

\begin{figure}[]
\includegraphics[scale=0.5]{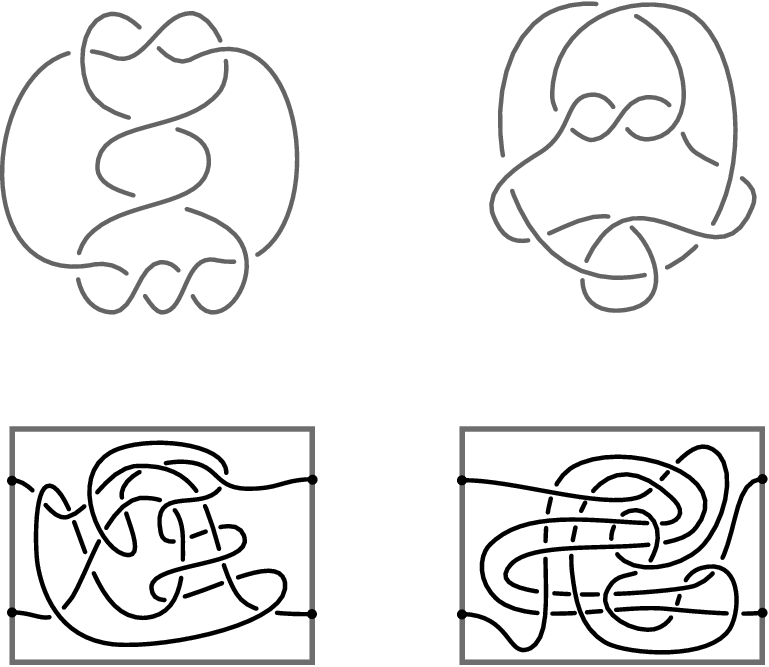} 
\caption{Two strong inversions $h_1$ and $h_2$ (left and right) on the knot $K=9_9$ with the relevant quotient tangle in each case. Note that according to Sakuma $\eta(K,h_1)=\eta(K,h_2)=-2t^{-2}+4-2t^2$ \cite{Sakuma1986}.}\label{fig:9_9}\end{figure}

\begin{figure}[]
\includegraphics[scale=0.5]{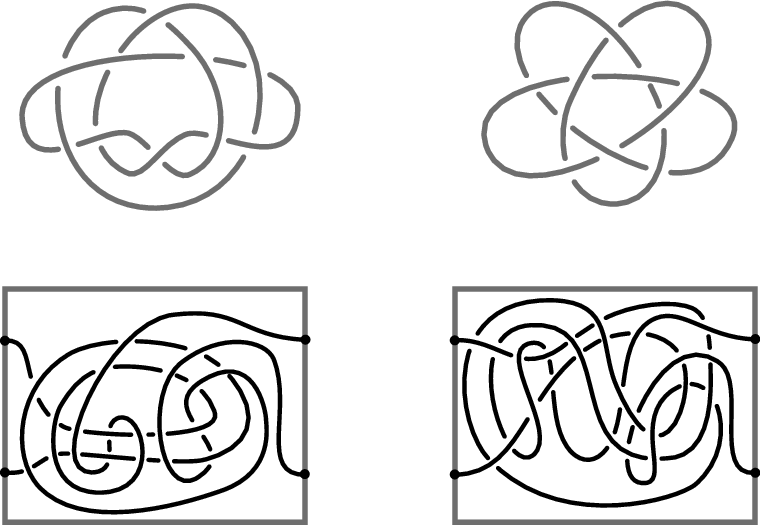} 
\caption{Two strong inversions $h_1$ and $h_2$ (left and right) on the knot $K=10_{155}$ with the relevant quotient tangle in each case. Note that according to Sakuma $\eta(K,h_1)=0$ and $\eta(K,h_2)=2t^{-2}-4+2t^2$ \cite{Sakuma1986}.}\label{fig:10_155}\end{figure}

For the first example the underlying knot is $9_9$. This knot admits a pair of strong inversions $h_1$ and $h_2$ (see Figure \ref{fig:9_9}) and is noteworthy as Sakuma's invariant fails to separate $(9_9,h_1)$ and $(9_9,h_2)$. We compute:

\[\begin{tikzpicture}[scale=0.85]

\draw[step=.5,black] (0,0) grid (9.5,1.5);

\node at (-2,1.25) {$\kh(9_9,h_1)$};

\node at (0.25,1.25) {$1$}; 
\node at (0.75,1.25) {$2$}; 
\node at (1.25,1.25) {$2$}; 
\node at (1.75,1.25) {$3$}; 
\node at (2.25,1.25) {$5$}; 
\node at (2.75,1.25) {$5$}; 
\node at (3.25,1.25) {$5$}; 
\node at (3.75,1.25) {$6$}; 
\node at (4.25,1.25) {$6$}; 
\node at (4.75,1.25) {$5$}; 
\node at (5.25,1.25) {$6$}; 
\node at (5.75,1.25) {$5$}; 
\node at (6.25,1.25) {$3$}; 
\node at (6.75,1.25) {$3$}; 
\node at (7.25,1.25) {$2$}; 
\node at (7.75,1.25) {$1$};

\node at (-2,0.25) {$\kh(9_9,h_2)$};

\node at (2.25,0.25) {$1$}; 
\node at (2.75,0.25) {$1$}; 
\node at (3.25,0.25) {$1$}; 
\node at (3.75,0.25) {$3$}; 
\node at (4.25,0.25) {$5$}; 
\node at (4.75,0.25) {$5$}; 
\node at (5.25,0.25) {$7$}; 
\node at (5.75,0.25) {$8$}; 
\node at (6.25,0.25) {$7$}; 
\node at (6.75,0.25) {$7$}; 
\node at (7.25,0.25) {$6$}; 
\node at (7.75,0.25) {$4$}; 
\node at (8.25,0.25) {$2$};
\node at (8.75,0.25) {$2$};
\node at (9.25,0.25) {$1$};

\node at (0.25,-0.25) {\footnotesize{-$8$}}; 
\node at (1.25,-0.25) {\footnotesize{-$6$}}; 
\node at (2.25,-0.25) {\footnotesize{-$4$}}; 
\node at (3.25,-0.25) {\footnotesize{-$2$}}; 
\node at (4.25,-0.25) {\footnotesize{$0$}}; 
\node at (5.25,-0.25) {\footnotesize{$2$}}; 
\node at (6.25,-0.25) {\footnotesize{$4$}}; 

\node at (7.25,-0.25) {\footnotesize{$6$}}; 
\node at (8.25,-0.25) {\footnotesize{$8$}}; 
\node at (9.25,-0.25) {\footnotesize{$10$}};

\draw [white, fill=white] (-0.015,0.515) rectangle (9.515,0.985);

\end{tikzpicture}\]

This example indicates that strong inversions need not be explicit on a given diagram, highlighting a subtlety in separating conjugacy classes. In fact, as pointed out by Paoluzzi \cite{Paoluzzi2005}, the fixed point sets of two strong inversions can be linked in $S^3$ in interesting ways. For example, $10_{155}$ admits a pair of strong inversions $h_1$ and $h_2$ (see Figure \ref{fig:10_155}) for which $\Fix(h_1)\cup\Fix(h_2)$ form a Hopf link. This is not made apparent in our diagrams; see \cite[Figure 10]{Paoluzzi2005} or \cite[Figure 3.1 (b)]{Sakuma1986}.  For this example we calculate:

\[
\begin{tikzpicture}[scale=0.85]

\draw[step=.5,black] (0,0) grid (7.5,1.5);
 
\node at (-2,.75) {$\kh(10_{155},h_1)$};

\node at (0.25,1.25) {$1$}; 
\node at (0.75,1.25) {$2$}; 
\node at (1.25,1.25) {$2$}; 
\node at (1.75,1.25) {$3$}; 
\node at (2.25,1.25) {$3$}; 
\node at (2.75,1.25) {$2$}; 
\node at (3.25,1.25) {$2$}; 
\node at (3.75,1.25) {$1$}; 

\node at (2.25,0.75) {$2$}; 
\node at (2.75,0.75) {$3$}; 
\node at (3.25,0.75) {$3$}; 
\node at (3.75,0.75) {$5$}; 
\node at (4.25,0.75) {$4$}; 
\node at (4.75,0.75) {$3$}; 
\node at (5.25,0.75) {$3$}; 
\node at (5.75,0.75) {$1$}; 

\node at (4.25,0.25) {$1$}; 
\node at (4.75,0.25) {$1$}; 
\node at (5.25,0.25) {$1$}; 
\node at (5.75,0.25) {$2$}; 
\node at (6.25,0.25) {$1$}; 
\node at (6.75,0.25) {$1$}; 
\node at (7.25,0.25) {$1$}; 
\end{tikzpicture}\]

\[\begin{tikzpicture}[scale=0.85]

\draw[step=.5,black] (0,0) grid (10.5,2.5);
 
\node at (-2,1.25) {$\kh(10_{155},h_2)$};

\node at (0.25,2.25) {$1$}; 
\node at (0.75,2.25) {$1$}; 
\node at (1.25,2.25) {$1$}; 
\node at (1.75,2.25) {$2$}; 
\node at (2.25,2.25) {$1$}; 
\node at (2.75,2.25) {$1$}; 
\node at (3.25,2.25) {$1$}; 

\node at (1.75,1.75) {$1$}; 
\node at (2.25,1.75) {$2$}; 
\node at (2.75,1.75) {$3$}; 
\node at (3.25,1.75) {$3$}; 
\node at (3.75,1.75) {$4$}; 
\node at (4.25,1.75) {$3$}; 
\node at (4.75,1.75) {$2$}; 
\node at (5.25,1.75) {$2$}; 

\node at (3.75,1.25) {$1$}; 
\node at (4.25,1.25) {$1$}; 
\node at (4.75,1.25) {$2$}; 
\node at (5.25,1.25) {$2$}; 
\node at (5.75,1.25) {$2$}; 
\node at (6.25,1.25) {$2$}; 
\node at (6.75,1.25) {$1$}; 
\node at (7.25,1.25) {$1$}; 

\node at (5.75,.75) {$1$}; 

\node at (6.75,.75) {$1$}; 
\node at (7.25,.75) {$1$}; 

\node at (8.25,0.75) {$1$}; 

\node at (7.75,.25) {$1$}; 

\node at (8.75,0.25) {$1$}; 
\node at (9.25,.25) {$1$}; 

\node at (10.25,0.25) {$1$};

\end{tikzpicture}\]

As illustration, we have omitted the absolute $\bZ$-grading and recorded instead $\kh(10_{155},h_1)$ and $\kh(10_{155},h_2)$ as relatively $(\bZ\times\bZ)$-graded vector spaces. This highlights considerable additional structure.

These (and other) examples point to an obvious question:

\begin{question}\label{qst:seperate}Does the invariant $\kh$ separate conjugacy classes of strong inversions in $\Sym(S^3,K)$ for a given prime knot? That is, given a prime knot $K$ admitting strong inversions $h$ and $h'$ is it the case that $(K,h)\simeq(K,h')$ if and only if $\kh(K,h)\cong\kh(K,h')$ as graded vector spaces? \end{question}

\begin{remark}\label{rmk:fig8} The emphasis on the grading is essential in this question: The figure eight admits a pair of strong inversions that are not distinguished by $\kh$ as relatively graded vector spaces but are distinguished by the absolute grading; see Section \ref{sec:bi}.\end{remark}

\section{Detecting non-amphicheirality}\label{sec:amph}

Recall that a knot $K$ is amphicheiral if $K\simeq K^*$, where $K^*$ denotes the mirror image of $K$, obtained by reversing orientation on $S^3$. For strongly invertible knots we write $(K,h)^*$; notice that if $h$ is unique then $(K,h)^*\simeq (K^*,h)$ makes sense (compare Section \ref{sub:mirror}). Regarding amphicheirality, Sakuma observes the following:

\begin{proposition}[Sakuma {\cite[Proposition 3.4 (1)]{Sakuma1986}}]
Let $K$ be an amphicheiral knot and suppose that $h$ is a unique strong inversion on $K$ (up to conjugacy in $\Sym(S^3,K)$). Then $(K,h)\simeq(K^*,h)$ and $\eta(K,h)$ vanishes. $\ensuremath\hfill\Box$
\end{proposition} 

Sakuma points out that for all but two strongly invertible hyperbolic knots with 9 or fewer crossings non-amphicheirality is detected by this condition (or a closely related condition \cite[Proposition 3.4 (2)]{Sakuma1986}). The exceptions are $8_{20}$ and $9_{40}$. The latter has non-zero signature ruling out amphicheirality,  but the former has vanishing signature and Sakuma invariant. 

Despite the fact that the amphicheirality of the tabulated knots is well established \cite{Perko1974}, this does raise an interesting question about the nature of algebraic invariants capable of detecting this subtle property. Along these lines, the non-vanishing result established in Theorem \ref{thm:unknot} suggests  that $\kh(K,h)$ is a good candidate invariant for detecting non-amphicheirality. We have:

\begin{proposition}\label{prp:amph}
Let $K$ be an amphicheiral knot and suppose that $h$ is a unique strong inversion on $K$ (up to conjugacy in $\Sym(S^3,K)$). Then $(K,h)\simeq(K^*,h)$ and $\kh(K,h)\cong\kh(K^*,h)$ as graded vector spaces. 
\end{proposition} 

\begin{proof} This follows immediately from Proposition \ref{prp:mirror}.\end{proof}

\begin{figure}[]
\raisebox{-20pt}{\includegraphics[scale=0.5]{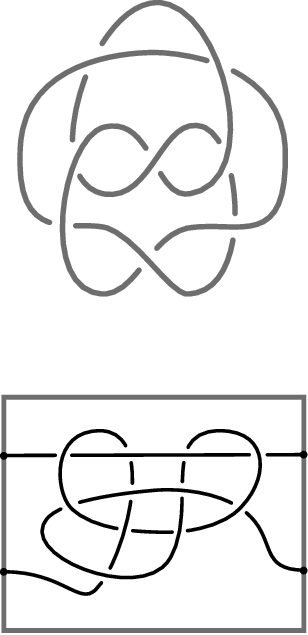}} 
\qquad
\begin{tikzpicture}[scale=0.85]

\draw[step=.5,black] (0,0) grid (5.5,.5);


\node at (0.25,.25) {$1$}; 
\node at (0.75,.25) {$1$}; 
\node at (1.25,.25) {$1$}; 
\node at (1.75,.25) {$2$}; 
\node at (2.25,.25) {$2$}; 
\node at (2.75,.25) {$2$}; 
\node at (3.25,.25) {$2$}; 
\node at (3.75,.25) {$2$}; 
\node at (4.25,.25) {$1$}; 
\node at (4.75,.25) {$1$}; 
\node at (5.25,.25) {$1$}; 

\node at (0.25,-0.25) {\footnotesize{-$6$}}; 
\node at (0.75,-0.25) {\footnotesize{-$5$}}; 
\node at (1.25,-0.25) {\footnotesize{-$4$}}; 
\node at (1.75,-0.25) {\footnotesize{-$3$}}; 
\node at (2.25,-0.25) {\footnotesize{-$2$}}; 
\node at (2.75,-0.25) {\footnotesize{-$1$}}; 
\node at (3.25,-0.25) {\footnotesize{$0$}}; 
\node at (3.75,-0.25) {\footnotesize{$1$}}; 
\node at (4.25,-0.25) {\footnotesize{$2$}}; 
\node at (4.75,-0.25) {\footnotesize{$3$}}; 
\node at (5.25,-0.25) {\footnotesize{$4$}}; 

\end{tikzpicture}

\caption{The knot $8_{20}$ with quotient tangle associated with the unique strong inversion $h$ and $\bZ$-graded vector space $\kh(8_{20},h)$. Note that $\eta(K,h)=0$ \cite{Sakuma1986}.}\label{fig:8_20}\end{figure}
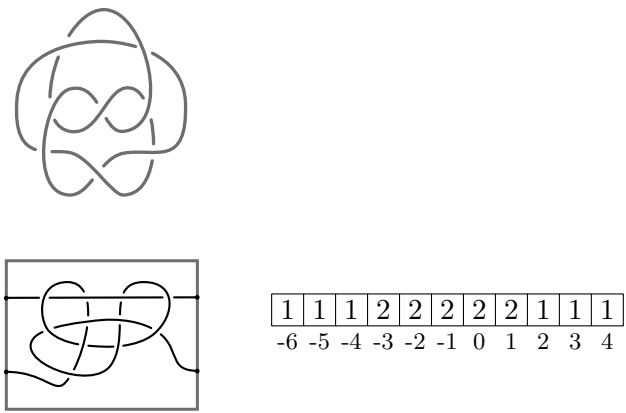

For example, this criterion detects the non-amphicheirality of $8_{20}$ while Sakuma's invariant does not. The calculation is summarized in Figure \ref{fig:8_20} (for unicity of $h$ we refer to Hartley \cite{Hartley1981}; see also Kodama and Sakuma \cite{KS1992}). However, it is well-known that the Jones polynomial of an amphicheiral knot is symmetric, and this gives a quick certification that $8_{20}$ is not amphicheiral. 

More generally, the Jones polynomial is typically very good at detecting non-amphicheirality. Given the relationship between the Jones polynomial and Khovanov homology, perhaps this should be explored further to ensure that we are not reinventing the wheel, in particular, that the information in $\kh(K,h)$ is not just a complicated repackaging of data from $\Khred(K)$. 

\begin{definition}
A knot $K$ is {J-amphicheiral} if the Jones polynomial of $K$ satisfies $V_K(t) = \sum_{i\ge0} a_i(t^i+t^{-i})$  for $a_i\in\bZ$.
\end{definition}

For example, the knot $9_{42}$ is J-amphicheiral. It is strongly invertible with a unique strong inversion, and both $\kh$ and $\eta$ may be used to establish non-amphicheirality. This knot also has non-zero signature giving an alternate (and much easier) means of confirming this fact and motivating a second definition.

\begin{definition}\label{def:Qamph}
A knot $K$ is {quasi-amphicheiral} if it is J-amphicheiral and has vanishing signature.\end{definition}
 
Amphicheiral knots are necessarily quasi-amphicheiral, however, quasi-amphicheiral knots that are non-amphicheiral seem to be quite rare. There are none with fewer than 9 crossings, for example; there are precisely three examples with 10 crossings: $10_{48},10_{71},10_{104}$. These are all thin knots, that is, the Khovanov homology $\Khred(K)$ is supported in a single diagonal for $K\in\{10_{48},10_{71},10_{104}\}$. Indeed, $\Khred(K)$ is determined by $V_K(t)$ and $\sigma(K)$ due to each of these knots being alternating. As a result, Khovanov homology does not detect the non-amphicheirality of these quasi-amphicheiral knots either, though in principle Khovanov homology should be more sensitive in this regard than the Jones polynomial (in fact, $9_{42}$ is an example supporting this presumption). Interestingly, $10_{71}$ and $10_{104}$ are distinct alternating knots that have identical invariants (Jones polynomial, signature and Khovanov homology).   

\begin{proposition}
Each $K\in\{10_{48},10_{71},10_{104}\}$ admits a unique strong inversion. 
\end{proposition}

\begin{proof}A strong inversion on each of the three knots is illustrated in Figure \ref{fig:fake}. Hartley proves that none of these knots admits a free period symmetry \cite{Hartley1981}. Cyclic symmetries are ruled out by Kodama and Sakuma, see in particular \cite[Table 3.1]{KS1992}. As each knot is hyperbolic $h$ must be unique; see Theorem \ref{thm:dihedral}. \end{proof}

As in previous examples, we will omit the unique strong inversion from the notation. This set of knots allows us to establish that $\kh(K)$ contains different information than $\{V_K(t), \sigma(K)\}$ and indeed $\Khred(K)$. By direct calculation we have:

\begin{figure}[]
\includegraphics[scale=0.5]{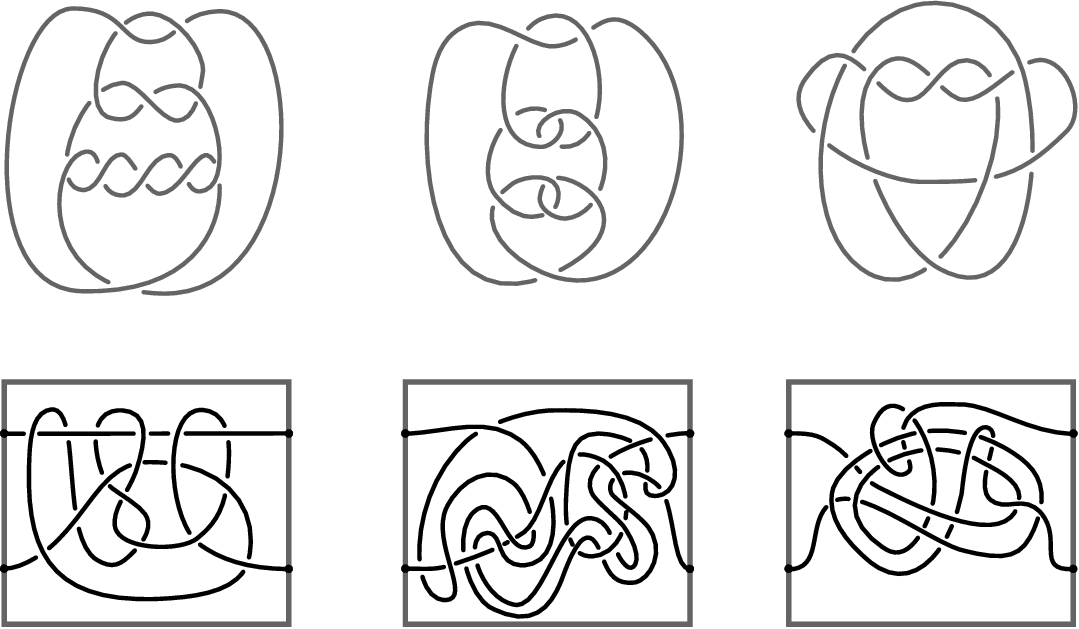} 
\caption{The quasi-amphicheiral knots $10_{48}$,  $10_{71}$ and $10_{104}$ (left-to-right, top) and their respective quotient tangles $T(10_{48})$,  $T(10_{71})$ and $T(10_{104})$ (left-to-right, bottom) corresponding to the unique strong inversion on each knot.}\label{fig:fake}\end{figure}

\begin{theorem}\label{thm:ten}
For \[K\in\{10_{48},10_{71},10_{104}\}\] the invariant $\kh(K)$ detects the non-amphicheirality of $K$. Moreover,   $\kh(10_{71})\ncong\kh(10_{104})$ distinguishing this pair, despite the fact that $\Khred(10_{71})\cong\Khred(10_{104})$.
\end{theorem}

\begin{proof}The calculations that, together with Proposition \ref{prp:amph}, establish Theorem \ref{thm:ten} are summarized as follows:

\[\begin{tikzpicture}[scale=0.85]

\draw[step=.5,black] (0,0) grid (16,2.5);

\node at (-0.85,2.25) {$\kh(10_{48})$};

\node at (4.25,2.25) {$1$}; 
\node at (4.75,2.25) {$3$}; 
\node at (5.25,2.25) {$4$}; 
\node at (5.75,2.25) {$5$}; 
\node at (6.25,2.25) {$8$}; 
\node at (6.75,2.25) {$10$}; 
\node at (7.25,2.25) {$10$}; 
\node at (7.75,2.25) {$11$}; 
\node at (8.25,2.25) {$11$}; 
\node at (8.75,2.25) {$9$}; 
\node at (9.25,2.25) {$8$}; 
\node at (9.75,2.25) {$7$}; 
\node at (10.25,2.25) {$4$}; 
\node at (10.75,2.25) {$2$}; 
\node at (11.25,2.25) {$1$};
\node at (11.75,2.25) {$1$};
 
\node at (-0.85,1.25) {$\kh(10_{71})$};

\node at (0.25,1.25) {$1$}; 
\node at (0.75,1.25) {$2$}; 
\node at (1.25,1.25) {$2$}; 
\node at (1.75,1.25) {$3$}; 
\node at (2.25,1.25) {$6$}; 
\node at (2.75,1.25) {$8$}; 
\node at (3.25,1.25) {$7$}; 
\node at (3.75,1.25) {$10$}; 
\node at (4.25,1.25) {$12$}; 
\node at (4.75,1.25) {$12$}; 
\node at (5.25,1.25) {$12$}; 
\node at (5.75,1.25) {$13$}; 
\node at (6.25,1.25) {$12$}; 
\node at (6.75,1.25) {$10$}; 
\node at (7.25,1.25) {$11$}; 
\node at (7.75,1.25) {$9$}; 
\node at (8.25,1.25) {$6$}; 
\node at (8.75,1.25) {$5$}; 
\node at (9.25,1.25) {$5$}; 
\node at (9.75,1.25) {$3$}; 
\node at (10.25,1.25) {$1$}; 
\node at (10.75,1.25) {$1$}; 
\node at (11.25,1.25) {$1$};

\node at (-0.85,0.25) {$\kh(\! 10_{104}\!)$};

\node at (5.25,0.25) {$1$}; 
\node at (5.75,0.25) {$1$}; 
\node at (6.25,0.25) {$1$}; 
\node at (6.75,0.25) {$3$}; 
\node at (7.25,0.25) {$5$}; 
\node at (7.75,0.25) {$4$}; 
\node at (8.25,0.25) {$7$}; 
\node at (8.75,0.25) {$10$}; 
\node at (9.25,0.25) {$10$}; 
\node at (9.75,0.25) {$11$}; 
\node at (10.25,0.25) {$14$}; 
\node at (10.75,0.25) {$14$}; 
\node at (11.25,0.25) {$12$};
\node at (11.75,0.25) {$14$};

\node at (12.25,0.25) {$12$};
\node at (12.75,0.25) {$9$};
\node at (13.25,0.25) {$8$};
\node at (13.75,0.25) {$7$};
\node at (14.25,0.25) {$4$};
\node at (14.75,0.25) {$2$};
\node at (15.25,0.25) {$2$};
\node at (15.75,0.25) {$1$};

\node at (0.25,-0.25) {\footnotesize{-$14$}}; 
\node at (1.25,-0.25) {\footnotesize{-$12$}}; 
\node at (2.25,-0.25) {\footnotesize{-$10$}}; 
\node at (3.25,-0.25) {\footnotesize{-$8$}}; 
\node at (4.25,-0.25) {\footnotesize{-$6$}}; 
\node at (5.25,-0.25) {\footnotesize{-$4$}}; 
\node at (6.25,-0.25) {\footnotesize{-$2$}}; 

\node at (7.25,-0.25) {\footnotesize{$0$}}; 
\node at (8.25,-0.25) {\footnotesize{$2$}}; 
\node at (9.25,-0.25) {\footnotesize{$4$}}; 
\node at (10.25,-0.25) {\footnotesize{$6$}}; 
\node at (11.25,-0.25) {\footnotesize{$8$}}; 
\node at (12.25,-0.25) {\footnotesize{$10$}}; 
\node at (13.25,-0.25) {\footnotesize{$12$}}; 
\node at (14.25,-0.25) {\footnotesize{$14$}}; 
\node at (15.25,-0.25) {\footnotesize{$16$}}; 

\draw [white, fill=white] (-0.015,0.515) rectangle (16.015,0.985);
\draw [white, fill=white] (-0.015,1.515) rectangle (16.015,1.985);

\end{tikzpicture}\]

 None of these vector spaces exhibit the requisite symmetry for amphicheirality. \end{proof}

Note that Theorem \ref{thm:summary} is an immediate corollary of Theorem \ref{thm:ten}. We emphasise that $\kh(10_{71})\ncong\kh(10_{104})$ as graded vector spaces. Indeed, $\dim\kh(10_{71})=\dim\kh(10_{104})=152$; compare Remark \ref{rmk:closing}. 

\begin{remark} Sakuma computes $\eta(10_{104},h)=-t^{-3}+t^{-2}-t^{-1}+2-t+t^2-t^3$ \cite[Example 3.5]{Sakuma1986}, the non-vanishing of which provides another means of verifying the non-amphicheirality of $10_{104}$. \end{remark}

\section{Conjectures}\label{sec:conjectures}

\subsection{Structural observations}

Consider the graded vector space \[V=\bF^{(0,\delta)}\oplus\bF^{(2,\delta)}\oplus\bF^{(3,\delta)}\oplus\bF^{(5,\delta)}\] for some $\delta\in\bZ$, where the second grading should be regarded as a relative $\bZ$-grading (compare the form of $\kh(K,h)$ in Section \ref{sec:examples} when $K$ is the trefoil). 

\begin{conjecture}\label{con:structure}For any strongly invertible knot $(K,h)$ there is a decomposition \[\kh(K,h)\cong\bigoplus_{i=1}^kV[m_i,n_i]\] as a $(\bZ\times\bZ)$-graded group (where the secondary grading is a relative grading) for pairs $(m_i,n_i)\in\bZ\times\bZ$. In particular, \[\dim\kh(K,h)\equiv 0\bmod{4}.\]\end{conjecture}

Note that a consequence of this conjecture is a Khovanov-theoretic alternative to the last step in the proof of Theorem \ref{thm:cable}. Namely, if $K$ is non-trivial, then so is $D(K)$ so that $\dim\kh(D(K))$ is purportedly at least 4 (combining Theorem \ref{thm:unknot} and Conjecture \ref{con:structure}) hence $\dim\Khred(\sC_2K)$ is at least 4 as well.

This conjecture is based only on empirical evidence from a range of calculations. While we have no explanation whatsoever for this surprisingly ordered behaviour, there is some precedent for this in the literature. For example  Lee's work \cite{Lee2005} (see also Rasmussen \cite{Rasmussen2010}) explained an observation of Bar-Natan \cite[Conjecture 1]{Bar-Natan2002}.  There is also a related conjecture that remains open due to Dunfield, Gukov and Rasmussen and an observed a three-step pairing \cite[Section 5.6, particularly Definition 5.5]{DGR2006}. Even if our conjecture proves to be incorrect there should be some  explanation for the observed behaviour on a wide range of examples. 

Also observed in examples (see Section \ref{sec:dist}) is the following.

\begin{conjecture}\label{con:rank} If $h_1$ and $h_2$ are strong inversions on a knot $K$ then $\dim\kh(K,h_1)=\dim\kh(K,h_2)$.\end{conjecture}

This again places emphasis on the graded structure of the invariant $\kh(K,h)$ (compare Question \ref{qst:seperate} and Remark \ref{rmk:fig8}).

Note that it is not the case that $\dim\Khred(L_1)=\dim\Khred(L_2)$ when $\Br_{L_1}\cong\Br_{L_2}$ \cite{Watson2010}, however this equality does hold on a surprising range of examples of three-manifolds that two-fold branch cover distinct links. Conjecture \ref{con:rank} would explain such an equality in the case where the three-manifold arises by Dehn surgery on a knot $K$ admitting a pair of strong inversions $h_1$ and $h_2$. In particular, for surgery coefficient $n$ we have branch sets $L_1=T^\circ_{K,h_1}(n)$ and $L_2=T^\circ_{K,h_2}(n)$ for the two-fold branched cover $S^3_n(K)$. 

\subsection{A Khovanov-theoretic characterisation of L-space knots} Recall that an L-space is a rational homology sphere $Y$ satisfying $\dim\HFhat(Y)=|H_1(Y;\bZ)|$, and a knot in $S^3$ admitting an L-space surgery is called an L-space knot \cite{OSz2005-lens}. This class of three-manifolds include lens spaces, for example. It is an interesting open problem to give a topological characterisation of L-spaces, and related to this is the problem of characterising L-space knots. In the presence of a strong inversion, we propose:

\begin{conjecture}\label{con:L-space}
A non-trivial knot $K$ admitting a strong inversion $h$ is an L-space knot if and only if $\kh(K,h)$ is supported in a single diagonal grading $\delta=u-q$. 
\end{conjecture}

Support for this conjecture may be found in \cite{Watson2011}: Any knot admitting a lens space surgery (compare Theorem \ref{thm:torus}) as well as the $(-2,3,q)$-pretzel knots satisfy the conjecture. It is also the case that given a knot satisfying the conjecture, all sufficiently positive cables of the knot will also satisfy the conjecture (see \cite[Theorem 6.1]{Watson2011}).  This follows from the observation that all of these examples are strongly invertible and admit a large surgery with a thin branch set (that is, the branch set has Khovanov homology supported in a single $\delta$-grading). 

We remark that it is implicit in the Berge conjecture that knots admitting a non-trivial lens space surgery must be strongly invertible. It is tempting to guess --- and indeed the original version of Conjecture \ref{con:L-space} did so! --- that this is a property of L-space knots in general, namely, that {\em L-space knots are strongly invertible}.  However  recent work of Baker and Luecke shows that this is not the case \cite{BakerLuecke}. Interestingly, their construction produces knots in $S^3$ with no symmetries at all but which admit surgeries that are two-fold branched covers of alternating knots. In particular, the surgery admits an involution and the associated branch set has thin Khovanov homology. 

\section{Afterward: An absolute bi-grading}\label{sec:bi}

In calculating $\kh(K,h)$ we have made essential use of the secondary grading $\delta$ on $\Khred(T_{K,h}(n))$. This is {\em a priori} a relative $\bZ$-grading so that $\kh(K,h)$ is naturally a $(\bZ\times\bZ)$-graded vector space (the second factor being the relative grading). It seems reasonable to attempt to promote (or, lift) this to an absolute bi-grading. To conclude, we will sketch a construction of such a lift. 

Let $T=T_{K,h}$ be the tangle associated with a given strongly invertible knot $(K,h)$. Notice that from Lemma \ref{lem:stability} we can take a sufficiently large $n$ so that $\Khred(T(n))$ is computed (by way of an iterated mapping cone) in terms of $\Khred(T(0))$ and $\bigoplus_{i=0}^nX[i,0]\cong \bigoplus_{i=0}^n\bF^{(u(i),\delta(i))}$ for some integer $u(i)$ and half-integer $\delta(i)$. 

Inspecting the proof of Lemma \ref{lem:stability} we see that $\delta(i)$ depends on the integer $n$, however this dependance disappears when $\delta$ is taken as a relative $\bZ$-grading instead of an absolute $\frac{1}{2}\bZ$-grading. In particular, we may fix a choice of absolute $\delta$ grading on $\kh(K,h)$ by requiring that the potential generators from $\bigoplus_{i=0}^nX[i,0]$ lie in $\delta = +1$. 

In the interest of preserving the symmetry under mirrors that was essential in application (see Section \ref{sec:amph}) it is more natural to fix $\delta=+\frac{1}{2}$ instead. It is only a cosmetic difference to fix $2\delta=+1$ to obtain an absolutely $(\bZ\times\odd)$-graded vector space (effectively clearing denominators in an {\em a priori} $(\bZ\times\half \bZ)$-graded vector space). As a result, for example, the trefoil (considered in our running example) is promoted to 

\[
\includegraphics[scale=0.5]{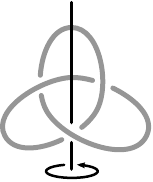}\qquad
\begin{tikzpicture}[scale=0.85]

\draw[step=.5,black] (0,0) grid (3,.5);

\node at (-0.375,0.25) {\footnotesize{$1$}};

\node at (0.25,.25) {$1$}; 

\node at (1.25,.25) {$1$}; 
\node at (1.75,.25) {$1$}; 

\node at (2.75,.25) {$1$}; 

\node at (0.25,-0.25) {\footnotesize{-$5$}}; 
\node at (0.75,-0.25) {\footnotesize{-$4$}}; 
\node at (1.25,-0.25) {\footnotesize{-$3$}}; 
\node at (1.75,-0.25) {\footnotesize{-$2$}}; 
\node at (2.25,-0.25) {\footnotesize{-$1$}}; 
\node at (2.75,-0.25) {\footnotesize{$0$}}; 

\end{tikzpicture}\]

as a $(\bZ\times\odd)$-graded vector space (the vertical axis represents $2\delta$, as in Figure \ref{fig:gradings}). Since the invariant for any torus knot is supported in a single $\delta$-grading (indeed, $2\delta=+1$ according to this absolute lift for positive torus knots; compare Theorem \ref{thm:torus}), this does not add too much new information. However, in general this does add considerably more structure. For example, the figure eight gives

\[
\includegraphics[scale=0.5]{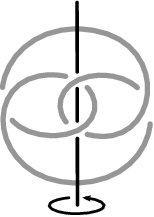}\qquad
\begin{tikzpicture}[scale=0.85]

\draw[step=.5,black] (0,0) grid (5,1);

\node at (-0.375,0.25) {\footnotesize{-$1$}};
\node at (-0.375,0.75) {\footnotesize{-$3$}};

\node at (0.25,.75) {$1$}; 

\node at (1.25,.75) {$1$}; 
\node at (1.75,.75) {$1$}; 

\node at (2.75,.75) {$1$}; 

\node at (2.25,.25) {$1$}; 

\node at (3.25,.25) {$1$}; 
\node at (3.75,.25) {$1$}; 

\node at (4.75,.25) {$1$}; 

\node at (0.25,-0.25) {\footnotesize{$0$}}; 
\node at (0.75,-0.25) {\footnotesize{$1$}}; 
\node at (1.25,-0.25) {\footnotesize{$2$}}; 
\node at (1.75,-0.25) {\footnotesize{$3$}}; 
\node at (2.25,-0.25) {\footnotesize{$4$}}; 
\node at (2.75,-0.25) {\footnotesize{$5$}}; 
\node at (3.25,-0.25) {\footnotesize{$6$}}; 
\node at (3.75,-0.25) {\footnotesize{$7$}}; 
\node at (4.25,-0.25) {\footnotesize{$8$}};
 \node at (4.75,-0.25) {\footnotesize{$9$}};

\end{tikzpicture}\]

(this may be extracted from the calculations in \cite{Watson2013}). 

\begin{remark}Notice that the figure eight admits a second strong inversion and,  since this knot is amphicheiral and hyperbolic, the pair of strong inversions $h_1$ and $h_2$ must be interchanged under mirror image \cite[Proposition 3.4 (2)]{Sakuma1986}. That is, $(K,h_1)^*\simeq(K,h_2)$ as strongly invertible knots where $K$ is the figure eight. In particular, this example illustrates the necessity for the $\bZ$-graded information in distinguishing strong inversion: $\kh(K,h_1)\cong\kh(K,h_2)$ even as relatively $(\bZ\times\bZ)$-graded groups (compare Remark \ref{rmk:fig8}).\end{remark}

A key feature of this choice --- building on Proposition \ref{prp:mirror} --- is summarized in the following statement, the proof of which is left to the reader.

\begin{proposition}
Let $(K,h)$ be a strongly invertible knot with strongly invertible mirror $(K,h)^*$. Then $\kh^{u,2\delta}(K,h)^*\cong \kh^{-u,-2\delta}(K,h)$ as $(\bZ\times\odd)$-graded vector spaces.  $\hfill\ensuremath\Box$
\end{proposition}

We note that this $(\bZ\times\odd)$-graded invariant of strong inversions typically contains considerably more information than its $\bZ$-graded counterpart. For example, revisiting the quasi-amphicheiral knots of Section \ref{sec:amph} we have:

\[
\begin{tikzpicture}[scale=0.85]

\draw[step=.5,black] (0,0) grid (8,1.5);
 
\node at (-2,.75) {$\kh(10_{48})$};

\node at (0.25,1.25) {$1$}; 
\node at (0.75,1.25) {$3$}; 
\node at (1.25,1.25) {$4$}; 
\node at (1.75,1.25) {$5$}; 
\node at (2.25,1.25) {$6$}; 
\node at (2.75,1.25) {$5$}; 
\node at (3.25,1.25) {$4$}; 
\node at (3.75,1.25) {$3$}; 
\node at (4.25,1.25) {$1$};

\node at (2.25,0.75) {$2$}; 
\node at (2.75,0.75) {$5$}; 
\node at (3.25,0.75) {$6$}; 
\node at (3.75,0.75) {$8$}; 
\node at (4.25,0.75) {$9$}; 
\node at (4.75,0.75) {$7$}; 
\node at (5.25,0.75) {$6$}; 
\node at (5.75,0.75) {$4$}; 
\node at (6.25,0.75) {$1$}; 

\node at (4.25,0.25) {$1$}; 
\node at (4.75,0.25) {$2$}; 
\node at (5.25,0.25) {$2$}; 
\node at (5.75,0.25) {$3$}; 
\node at (6.25,0.25) {$3$}; 
\node at (6.75,0.25) {$2$}; 
\node at (7.25,0.25) {$2$}; 
\node at (7.75,0.25) {$1$}; 

\node at (-0.375,1.25) {\footnotesize{-$1$}}; 
\node at (-0.375,0.75) {\footnotesize{$1$}}; 
\node at (-0.375,0.25) {\footnotesize{$3$}}; 

\node at (0.25,-0.25) {\footnotesize{-$6$}}; 
\node at (0.75,-0.25) {\footnotesize{-$5$}}; 
\node at (1.25,-0.25) {\footnotesize{-$4$}}; 
\node at (1.75,-0.25) {\footnotesize{-$3$}}; 
\node at (2.25,-0.25) {\footnotesize{-$2$}}; 
\node at (2.75,-0.25) {\footnotesize{-$1$}}; 
\node at (3.25,-0.25) {\footnotesize{$0$}}; 
\node at (3.75,-0.25) {\footnotesize{$1$}}; 
\node at (4.25,-0.25) {\footnotesize{$2$}}; 
\node at (4.75,-0.25) {\footnotesize{$3$}}; 
\node at (5.25,-0.25) {\footnotesize{$4$}}; 
\node at (5.75,-0.25) {\footnotesize{$5$}}; 
\node at (6.25,-0.25) {\footnotesize{$6$}}; 
\node at (6.75,-0.25) {\footnotesize{$7$}}; 
\node at (7.25,-0.25) {\footnotesize{$8$}}; 
\node at (7.75,-0.25) {\footnotesize{$9$}}; 
\end{tikzpicture}\]

\[\begin{tikzpicture}[scale=0.85]

\draw[step=.5,black] (0,0) grid (11.5,2.5);
 
\node at (-2,1.25) {$\kh(10_{71})$};

\node at (0.25,2.25) {$1$}; 
\node at (0.75,2.25) {$2$}; 
\node at (1.25,2.25) {$2$}; 
\node at (1.75,2.25) {$3$}; 
\node at (2.25,2.25) {$3$}; 
\node at (2.75,2.25) {$2$}; 
\node at (3.25,2.25) {$2$}; 
\node at (3.75,2.25) {$1$}; 

\node at (2.25,1.75) {$3$}; 
\node at (2.75,1.75) {$6$}; 
\node at (3.25,1.75) {$5$}; 
\node at (3.75,1.75) {$9$}; 
\node at (4.25,1.75) {$8$}; 
\node at (4.75,1.75) {$5$}; 
\node at (5.25,1.75) {$6$}; 
\node at (5.75,1.75) {$2$}; 

\node at (4.25,1.25) {$4$}; 
\node at (4.75,1.25) {$7$}; 
\node at (5.25,1.25) {$6$}; 
\node at (5.75,1.25) {$11$}; 
\node at (6.25,1.25) {$9$}; 
\node at (6.75,1.25) {$6$}; 
\node at (7.25,1.25) {$7$}; 
\node at (7.75,1.25) {$2$};

\node at (6.25,0.75) {$3$}; 
\node at (6.75,0.75) {$4$}; 
\node at (7.25,0.75) {$4$}; 
\node at (7.75,0.75) {$7$}; 
\node at (8.25,0.75) {$5$}; 
\node at (8.75,0.75) {$4$}; 
\node at (9.25,0.75) {$4$}; 
\node at (9.75,0.75) {$1$};

\node at (8.25,0.25) {$1$}; 
\node at (8.75,0.25) {$1$}; 
\node at (9.25,0.25) {$1$}; 
\node at (9.75,0.25) {$2$}; 
\node at (10.25,0.25) {$1$}; 
\node at (10.75,0.25) {$1$}; 
\node at (11.25,0.25) {$1$};

\node at (-0.375,2.25) {\footnotesize{-$1$}}; 
\node at (-0.375,1.75) {\footnotesize{$1$}}; 
\node at (-0.375,1.25) {\footnotesize{$3$}}; 
\node at (-0.375,0.75) {\footnotesize{$5$}}; 
\node at (-0.375,0.25) {\footnotesize{$7$}}; 

\node at (0.25,-0.175) {\footnotesize{-$14$}}; 
\node at (0.75,-0.375) {\footnotesize{-$13$}}; 
\node at (1.25,-0.175) {\footnotesize{-$12$}}; 
\node at (1.75,-0.375) {\footnotesize{-$11$}}; 
\node at (2.25,-0.175) {\footnotesize{-$10$}}; 
\node at (2.75,-0.25) {\footnotesize{-$9$}}; 
\node at (3.25,-0.25) {\footnotesize{-$8$}}; 
\node at (3.75,-0.25) {\footnotesize{-$7$}}; 
\node at (4.25,-0.25) {\footnotesize{-$6$}}; 
\node at (4.75,-0.25) {\footnotesize{-$5$}}; 
\node at (5.25,-0.25) {\footnotesize{-$4$}}; 
\node at (5.75,-0.25) {\footnotesize{-$3$}}; 
\node at (6.25,-0.25) {\footnotesize{-$2$}}; 
\node at (6.75,-0.25) {\footnotesize{-$1$}}; 
\node at (7.25,-0.25) {\footnotesize{$0$}}; 
\node at (7.75,-0.25) {\footnotesize{$1$}}; 
\node at (8.25,-0.25) {\footnotesize{$2$}}; 
\node at (8.75,-0.25) {\footnotesize{$3$}}; 
\node at (9.25,-0.25) {\footnotesize{$4$}}; 
\node at (9.75,-0.25) {\footnotesize{$5$}}; 
\node at (10.25,-0.25) {\footnotesize{$6$}}; 
\node at (10.75,-0.25) {\footnotesize{$7$}}; 
\node at (11.25,-0.25) {\footnotesize{$8$}}; 

\end{tikzpicture}\]

\[\begin{tikzpicture}[scale=0.85]

\draw[step=.5,black] (0,0) grid (11,2.5);
 
\node at (-2,1.25) {$\kh(10_{104})$};

\node at (0.25,2.25) {$1$}; 
\node at (0.75,2.25) {$1$}; 
\node at (1.25,2.25) {$1$}; 
\node at (1.75,2.25) {$2$}; 
\node at (2.25,2.25) {$1$}; 
\node at (2.75,2.25) {$1$}; 
\node at (3.25,2.25) {$1$}; 

\node at (1.75,1.75) {$1$}; 
\node at (2.25,1.75) {$4$}; 
\node at (2.75,1.75) {$3$}; 
\node at (3.25,1.75) {$5$}; 
\node at (3.75,1.75) {$6$}; 
\node at (4.25,1.75) {$3$}; 
\node at (4.75,1.75) {$4$}; 
\node at (5.25,1.75) {$2$}; 

\node at (3.25,1.25) {$1$}; 
\node at (3.75,1.25) {$4$}; 
\node at (4.25,1.25) {$7$}; 
\node at (4.75,1.25) {$7$}; 
\node at (5.25,1.25) {$10$}; 
\node at (5.75,1.25) {$9$}; 
\node at (6.25,1.25) {$6$}; 
\node at (6.75,1.25) {$6$}; 
\node at (7.25,1.25) {$2$};

\node at (5.25,0.75) {$2$}; 
\node at (5.75,0.75) {$5$}; 
\node at (6.25,0.75) {$6$}; 
\node at (6.75,0.75) {$8$}; 
\node at (7.25,0.75) {$9$}; 
\node at (7.75,0.75) {$7$}; 
\node at (8.25,0.75) {$6$}; 
\node at (8.75,0.75) {$4$}; 
\node at (9.25,0.75) {$1$}; 

\node at (7.25,0.25) {$1$}; 
\node at (7.75,0.25) {$2$}; 
\node at (8.25,0.25) {$2$}; 
\node at (8.75,0.25) {$3$}; 
\node at (9.25,0.25) {$3$}; 
\node at (9.75,0.25) {$2$}; 
\node at (10.25,0.25) {$2$}; 
\node at (10.75,0.25) {$1$};

\node at (-0.375,2.25) {\footnotesize{-$5$}}; 
\node at (-0.375,1.75) {\footnotesize{-$3$}}; 
\node at (-0.375,1.25) {\footnotesize{-$1$}}; 
\node at (-0.375,0.75) {\footnotesize{$1$}}; 
\node at (-0.375,0.25) {\footnotesize{$3$}}; 

\node at (0.25,-0.25) {\footnotesize{-$4$}}; 
\node at (0.75,-0.25) {\footnotesize{-$3$}}; 
\node at (1.25,-0.25) {\footnotesize{-$2$}}; 
\node at (1.75,-0.25) {\footnotesize{-$1$}}; 
\node at (2.25,-0.25) {\footnotesize{$0$}}; 
\node at (2.75,-0.25) {\footnotesize{$1$}}; 
\node at (3.25,-0.25) {\footnotesize{$2$}}; 
\node at (3.75,-0.25) {\footnotesize{$3$}}; 
\node at (4.25,-0.25) {\footnotesize{$4$}}; 
\node at (4.75,-0.25) {\footnotesize{$5$}}; 
\node at (5.25,-0.25) {\footnotesize{$6$}}; 
\node at (5.75,-0.25) {\footnotesize{$7$}}; 
\node at (6.25,-0.25) {\footnotesize{$8$}}; 
\node at (6.75,-0.25) {\footnotesize{$9$}}; 
\node at (7.25,-0.25) {\footnotesize{$10$}}; 
\node at (7.75,-0.25) {\footnotesize{$11$}}; 
\node at (8.25,-0.25) {\footnotesize{$12$}}; 
\node at (8.75,-0.25) {\footnotesize{$13$}}; 
\node at (9.25,-0.25) {\footnotesize{$14$}}; 
\node at (9.75,-0.25) {\footnotesize{$15$}}; 
\node at (10.25,-0.25) {\footnotesize{$16$}}; 
 \node at (10.75,-0.25) {\footnotesize{$17$}}; 

\end{tikzpicture}\]

As this is apparently stronger information than the integer grading used to this point, it would be interesting to exhibit, for example, a quasi-amphicheiral knot for which determining the non-amphicheirality depends on this additional structure.  

\begin{remark}\label{rmk:closing}
As observed in the proof of Theorem \ref{thm:ten}, $\dim\kh(10_{71})=\dim\kh(10_{104})$; consulting the invariants above the number of $\delta$-gradings (i.e. the homological width) supporting these invariants coincide. It is interesting that certain aspects of $\kh(10_{71})$ and $\kh(10_{104})$  (particularly, integer-valued invariants derived from $\kh$) coincide given that $\Khred(10_{71})\cong\Khred(10_{104})$. The fact that $\kh(10_{71})$ and $\kh(10_{104})$ differ as $\delta$-graded groups (absolutely or relatively) and are therefore separated by $\kh$ provides another application of the gradings in Khovanov homology to distinguish this pair.
\end{remark}

\begin{footnotesize}
\subsection*{Acknowledgements}  I would like to thank  Ciprian Manolescu and Luisa Paoluzzi for inspiring conversations and helpful comments.  This work took shape at the Simons Center for Geometry and Physics during the program {\em Symplectic and Contact Geometry and Connections to Low-Dimensional Topology}. I would like to thank the organizers for providing a great working environment. 
\end{footnotesize}

\bibliographystyle{plain}
\bibliography{involutions}

\end{document}